\documentclass[a4paper,10pt]{article}

\usepackage{
	amsmath, 
	amssymb, 
	amsthm,
  fullpage,
  hyperref,
  bm,
  ifthen,
  microtype,
  tikz,
}

\allowdisplaybreaks

\usepackage[USenglish]{babel}

\newcommand{\CC}{\mathrm{\mathbb{C}}}

\newcommand{\NN}{\mathrm{\mathbb{N}}}

\newcommand{\RR}{\mathrm{\mathbb{R}}}

\newcommand{\ZZ}{\mathrm{\mathbb{Z}}}
\newcommand{\GL}{\mathrm{GL}}

\newcommand{\id}{\mathrm{id}}

\newcommand{\tensor}{\otimes}

\newcommand{\diag}{\mathrm{diag}}
\newcommand{\End}{\mathrm{End}}
\newcommand{\tr}{\mathrm{tr}}

\newcommand{\Aut}{\mathrm{Aut}}
\newcommand{\rank}{\mathrm{rank}}
\newcommand{\qexp}{q\textnormal{-exp}}
\newcommand{\ad}{\mathrm{ad}}

\newcommand{\SU}{\mathrm{SU}}

\newcommand{\Uq}{\mathcal{U}_q}

\def\qbin#1#2{
  \left[ \genfrac{}{}{0pt}{}{#1}{#2} \right]
}

\theoremstyle{definition}
\newtheorem{theorem}{Theorem}[section]
\newtheorem{proposition}[theorem]{Proposition}
\newtheorem{lemma}[theorem]{Lemma}
\newtheorem{definition}[theorem]{Definition}
\newtheorem{corollary}[theorem]{Corollary}
\newtheorem{example}[theorem]{Example}
\newtheorem{remark}[theorem]{Remark}

\numberwithin{equation}{section}

\title{Radial part calculations for quantum symmetric pairs with simple generators}
\author{
  Noud Aldenhoven\footnote{Email address: \href{mailto:n.aldenhoven@math.ru.nl}{n.aldenhoven@math.ru.nl}}\\ 
  Radboud University Nijmegen, FNWI, IMAPP, \\
  Heyendaalseweg 135, 6525 AJ Nijmegen, The Netherlands
}

\begin{document}

\maketitle

\begin{abstract}
We introduce the class of quantum symmetric pairs with simple generators.
It is proved that the radial part of every element of a quantum symmetric pair with simple generators restricted to the set of regular points of this element can be computed.
These computations are done explicitly for the Casimir elements of the quantum analogues of $(\SU(2), \mathrm{U}(1))$, $(\SU(2) \times \SU(2), \diag)$ and $(\SU(3), \mathrm{U}(2))$ and give rise to second order $q$-difference equations for matrix valued spherical functions in general.
\end{abstract}

\section{Introduction}
Computing the radial part of Casimir elements of quantum symmetric pairs goes back to by Koornwinder \cite{Koo93} for the quantum analogue of $(\SU(2), \mathrm{U}(1))$.
Koornwinder \cite{Koo93} identified the action of the Casimir element of the quantized universal enveloping algebra of $\SU(2)$ with a second order $q$-difference equation for the Askey-Wilson polynomials in two parameters.
It turns out that the radial part of Casimir elements of quantum analogues of symmetric pairs, which generate the center of the quantized universal enveloping algebra, corresponds to $q$-difference equations for spherical functions.
Dijkhuizen, Noumi and Sugitani \cite{DN98, DS99, Nou96, NDS97, NS95, Sug99} continued Koornwinder's work and found many more examples of $q$-difference equations which have spherical functions as solutions.
They found many important non-trivial connections between the representations theory of quantum symmetric pairs and orthogonal polynomials appearing in the $q$-Askey scheme \cite{KLS10}.
Since the quantized universal enveloping algebra is a Hopf algebra, which does not contain many Hopf subalgebras, one of the main problems was to find a good analogue for the symmetric pairs $(G, K)$ for compact Lie groups $G$.
Letzter \cite{Let99, Let02, Let03, Let04, Let08} studied quantum symmetric pairs of the form $(\Uq(\mathfrak{g}), \mathcal{B})$, where $\Uq(\mathfrak{g})$ is the quantized universal enveloping algebra of Lie algebra $\mathfrak{g}$ and $\mathcal{B}$ a right coideal subalgebra of $\Uq(\mathfrak{g})$, i.e. $\Delta(\mathcal{B}) \subseteq \mathcal{B} \tensor \Uq(\mathfrak{g})$.
The quantum symmetric pairs turn out to be good analogues for the symmetric pairs $(G, K)$ of Lie groups.

If $\mathcal{A}$ is the quantum torus of $\Uq(\mathfrak{g})$, computing the scalar valued radial part of $Y \in \Uq(\mathfrak{g})$ boils down to computing $AY$ modulo the augmentation ideal $\mathcal{B}_+ = \{ B \in \mathcal{B} : \epsilon(B) = 0 \}$ for almost all $A \in \mathcal{A}$.
Note that we use the counit $\epsilon$ as a one-dimensional representation of the coideal subalgebra $\mathcal{B}$.
More recently Letzer \cite{Let03, Let04} united the computations for the scalar valued radial parts for all quantum symmetric pairs where the restricted root system is reduced. 
Letzter showed \cite[Theorem 8.2]{Let04} that the scalar valued radial part of the Casimir elements for the quantum symmetric pairs where the restricted root system is reduced  gives rise to $q$-difference equations for Macdonald polynomials.

However restricting the radial part to the scalar valued level, where only the trivial irreducible representation $\epsilon$ is taking into account, throws away information.
Therefore the problem remains to calculate the radial part of elements of the quantized universal enveloping algebra on the level of any finite dimensional irreducible representation of $\mathcal{B}$ in general.
For semi-simple Lie algebras this problem has been studied, e.g. Warner \cite[Chapter 8.2]{War72} for computations of the radial part of the center and Casselman and Mili\v{c}i\'c \cite{CM82} for computations of the radial part for any element of $\mathcal{U}(\mathfrak{g})$.
For a symmetric pair $(G, K)$ of Lie groups we have the Cartan decomposition $G = KAK$, see \cite[Chapter V, Theorem 6.7]{Hel78}, where $A$ is the torus of $G$.
Unfortunately, there does not exist a Cartan decomposition for the symmetric pairs $(\mathcal{U}(\mathfrak{g}), \mathcal{U}(\mathfrak{k}))$ of universal enveloping algebras in general.
Therefore Casselman and Mili\v{c}i\'c \cite{CM82} used the Iwasawa decomposition for $(\mathcal{U}(\mathfrak{g}), \mathcal{U}(\mathfrak{k}))$.
Using the Iwasawa decomposition Casselman and Mili\v{c}i\'c \cite{CM82} proved that there exists a map $\Pi$ which maps every element of $\mathcal{U}(\mathfrak{g})$ to a matrix valued differential equation defined on the regular points of the torus.
The map $\Pi$ is a useful tool for finding matrix valued differential equations for matrix valued spherical functions, see \cite[Theorem 3.1]{CM82} and \cite[Example 3.7]{CM82}.

Using \cite[Theorem 3.1]{CM82} Koelink, van Pruijssen and Rom\'an \cite{KvPR12, KvPR13} computed the matrix valued radial part for the two Casimir elements of $\mathcal{U}(\mathfrak{su}(2)) \tensor \mathcal{U}(\mathfrak{su}(2))$ related to the symmetric pair $(\SU(2) \times \SU(2), \diag)$.
Koelink, van Pruijssen and Rom\'an \cite{KvPR12, KvPR13} studied matrix valued spherical functions on $(\SU(2) \times \SU(2), \diag)$.
The matrix valued radial part gives rise to a first and second order differential equation.
These two differential equations \cite[Theorem 7.14]{KvPR13} are essential for completely classifying all matrix valued orthogonal polynomials related to the spherical functions of $(\SU(2) \times \SU(2), \diag)$, see \cite[Theorem 6.3]{KvPR13}.

In \cite{AKR15} matrix valued spherical functions on the quantum analogue of $(\SU(2) \times \SU(2), \diag)$ are studied.
The center of the quantum analogue of $\SU(2) \times \SU(2)$ is generated by two Casimir elements.
The radial part on the level of any finite dimensional irreducible representation of the quantized universal enveloping algebra is computed, see \cite[Proposition 5.10]{AKR15}.
The $q$-difference equations are essential to complete the matrix valued spherical functions, which in turn are related to matrix valued orthogonal polynomials, see \cite[Theorem 4.17]{AKR15}.

The method used in \cite{AKR15} to compute the radial part of the Casimir elements on the level of any finite dimensional irreducible representation is ad-hoc.
This paper shows that these computations can be extended to a subclass of the quantum symmetric pairs with so-called ``simple generators'' such that we can put these computations in a more general framework.
Moreover we give an explicit algorithm to compute the radial part of $AY$ for every $Y \in \Uq(\mathfrak{g})$ and $A \in \mathcal{A}$ such that $A$ is a regular point for $Y$.
This method is a $q$-analogue for Casselman and Mili\v{c}i\'c \cite{CM82}.
We apply this method to compute the radial part of the Casimir elements of the quantum analogue of $(\SU(2), \mathrm{U}(1))$, $(\SU(2) \times \SU(2), \diag)$ and $(\SU(3), \mathrm{U}(2))$.
Note that this method not only applies to central elements, but that it can be used to compute the radial part in general for every element of a quantum symmetric pair with simple generators.

However the problem of calculating the radial part of elements of quantum symmetric pairs in general still remains an open question.
To extend the algorithm described in this paper to quantum symmetric pairs in general we must find commutation relations between $F_i$ and $\theta_q(F_j K_j)$.
We were not able to find commutation relations that behave well enough to apply the method.
Therefore we will only consider quantum symmetric pairs with simple generators.

The article is organized as follows.
In Section \ref{q-radial_sec_quea} we fix the notation of quantized universal enveloping algebras on Kac-Moody algebras.
In Section \ref{q-radial_sec_symmetricquantumkacmoodypairs} the definition of quantum symmetric pairs with simple generators is given. 
The class of quantum symmetric pairs with simple generators are a subclass of quantum symmetric pairs introduced by Kolb \cite{Kol14}.
We give a complete classification of all quantum symmetric pairs with simple generators related to the symmetric pairs of semi-simple Lie algebras found in \cite{Ara62} and \cite{Hel78}.
In Section \ref{q-radial_sec_quantuminfcartandecomp} we prove a quantum analogue of the Iwasawa decomposition for quantum symmetric pairs with simple generators.
Then we state Theorem \ref{q-radial_thm_main} of the article, which is the main theorem of the paper.
The proof of Theorem \ref{q-radial_thm_main} explains how to calculate the radial part of $AZ$.
Theorem \ref{q-radial_thm_main} is technical and we will see in Section \ref{q-radial_sec_sphericalfunctions} that this theorem has important applications.
In Section \ref{q-radial_sec_sphericalfunctions} we study spherical functions on quantum symmetric pairs with simple generators.
We show that there exists a quantum analogue for the map $\Pi$ of Casselman and Mili\v{c}i\'c \cite{CM82}, see Definition \ref{q-radial_def_radial_part} and Theorem \ref{q-radial_thm_CM_analogue}.
At the end of Section \ref{q-radial_sec_sphericalfunctions} we study the $*$-invariance and state Theorem \ref{q-radial_thm_bi_invariant_weight} which proves an orthogonality relation for spherical functions.
In Section \ref{q-radial_sec_applications} we show that the proof of Theorem \ref{q-radial_thm_main} can be used to compute the radial part explicitly for the Casimir elements of the quantum analogues of $(\SU(2), \mathrm{U}(1))$, $(\SU(2) \times \SU(2), \diag)$ and $(\SU(3), \mathrm{U}(2))$.
The results for the quantum analogue of $(\SU(2), \mathrm{U}(1))$ match with the results of Koornwinder \cite{Koo93} when restricted to the trivial representation $\epsilon$.
Moreover, the radial part calculations of the center in general give an alternative proof for \cite[Theorem 7.6]{Koe96}.
The results for the quantum analogue of $(\SU(2) \times \SU(2), \diag)$ match with the results of \cite{AKR15}.
The radial part of the two second order Casimir elements, generating the center of the quantum analogue of $(\SU(3), \mathrm{U}(2))$, are calculated.
We identify the radial part of the center restricted to the trivial representation $\epsilon$ with Askey-Wilson polynomials in two free parameters, which match with the results of Dijkhuizen and Noumi \cite{DN98} for the quantum analogue of $(SU(3), \textrm{U}(2))$.
Note that the quantum analogue of $(\SU(3), \mathrm{U}(2))$ is excluded by Letzter \cite{Let04}, since the restricted root system is non-reduced.
Moreover, we compute the radial part of the center in general extending the result of Dijkhuizen en Noumi \cite{DN98} for the quantum analogue of $(\SU(3), \mathrm{U}(2))$ to the matrix valued case.

\section{Quantized universal enveloping algebra} \label{q-radial_sec_quea}
In this section we fix the notation.
For more information we refer to Kac \cite{Kac90}, Kolb \cite{Kol14} and Lusztig \cite{Lus94}.
We mainly follow Kolb \cite[\S 2.1 and \S 3.1]{Kol14}.

Let $I$ be a finite set, and let $A = (a_{ij})_{i,j \in I}$ be a generalized Cartan matrix, i.e. for all $i, j \in I$, $a_{ij} \in \ZZ$, $a_{ii} = 2$, $a_{ij} \leq 0$ for $i \neq j$ and $a_{ij} = 0$ if and only if $a_{ji} = 0$.
Assume that there is a $D = \diag(\epsilon_i : i \in I)$ with coprime entries $\epsilon_i \in \NN$ such that $DA$ is symmetric.
Define the dual weight lattice $P^{\vee}$ to be the free abelian group of rank $2|I| - \rank(A)$ generated over $\mathbb{Z}$ by $\{h_i : i \in I\}$ and $\{ d_s : s = 1, 2, \ldots, |I| - \rank(A) \}$.
Set $\mathfrak{h} = \mathbb{C} \tensor_{\mathbb{Z}} P^{\vee}$.
The weight lattice $P$ of $A$ defined by $P = \{ \lambda \in \mathfrak{h}^* : \lambda(P^{\vee}) \subseteq \ZZ \}$.
Define $\Pi^{\vee} = \{ h_i : i \in I \}$, choose $\alpha_i \in \mathfrak{h}^*$ linearly independent such that $\alpha_i(h_j) = a_{ji}$, $\alpha_i(d_s) \in \{0,1\}$ and take $\Pi = \{ \alpha_i : i \in I \}$.
Denote $Q = \ZZ \Pi$ for the root lattice and $Q^{\vee} = \ZZ \Pi^{\vee}$ for the coroot lattice of $Q$.
Choose the set $Q^+ = \sum_{i \in I} \NN \alpha_i$ of positive roots of $Q$.
We extend $P^{\vee}$, $Q^{\vee}$ and $Q$ to $\frac{1}{2}\mathbb{Z}$ by taking $\check{P}^{\vee} = \frac{1}{2}\ZZ[P^{\vee}]$, $\check{Q}^{\vee} = \frac{1}{2}\ZZ[\check{Q}^{\vee}]$ and $\check{Q} = \frac{1}{2}\ZZ[Q]$.
Introduce the bilinear form on $\mathfrak{h}^*$ by $(\alpha_i, \alpha_j) = \epsilon_i a_{ij}$.

The Kac-Moody algebra $\mathfrak{g} = \mathfrak{g}(A)$ is the Lie algebra over $\CC$ generated by $\mathfrak{h}$ and $e_i, f_i$ for $i \in I$ with relations given in \cite[\S 1.3]{Kac90}.
Let $\mathfrak{g}' = [\mathfrak{g}, \mathfrak{g}]$ be the derived Lie algebra and note that $\mathfrak{g}'$ is generated by $e_i, f_i$ for $i \in I$.

For any $i \in I$ define the fundamental reflections $r_i \in \GL(\mathfrak{h})$ by $r_i(h) = h - \alpha_i(h) h_i$ for all $h \in \mathfrak{h}$.
The Weyl group $W$ is generated by the fundamental reflections $r_i$.

Let $\CC(q)$ be the complex field of rational functions in an indeterminate $q$ over $\CC$.
The quantized enveloping algebra $\mathcal{U}_q(\mathfrak{g})$ of $\mathfrak{g}$ is the associative unital $\CC(q)$-algebra generated by $E_i$, $F_i$ and $K_{\mu}$ for $i \in I$ and $\mu \in P^{\vee}$, subjected to the relations
\begin{equation}
\label{q-radial_eqn_uqgrelations}
\begin{split}
K_0 &= 1, \quad K_{h} K_{h'} = K_{h + h'},  \\
K_{h} E_i &= q^{\alpha_i(h)} E_i K_{h}, \quad K_{h} F_i = q^{-\alpha_i(h)} F_{i} K_{h},  \\
[E_i, F_j] &= \delta_{i,j} \frac{K_i - K_i^{-1}}{q_i - q_i^{-1}}, \quad \text{where } q_i = q^{\epsilon_i}, K_i = K_{h_i}^{\epsilon_i},
\end{split}
\end{equation}
where $h, h' \in P^{\vee}$ and the quantum Serre relations $F_{ij}(E_i, E_j) = F_{ij}(F_i, F_j) = 0$ for all $i, j \in I$ where
\begin{equation*}
F_{ij}(X,Y) = \sum_{n=0}^{1-a_{ij}} (-1)^n \qbin{1-a_{ij}}{n}_{q_{i}} X^{1-a_{ij}-n} Y X^n.
\end{equation*}
The coproduct $\Delta$, counit $\epsilon$ and antipode $S$ on $\mathcal{U}_q(\mathfrak{g})$ are given by
\begin{equation*}
\begin{split}
\Delta : E_i, F_i, K_{h} &\mapsto E_i \tensor 1 + K_i \tensor E_i, F_i \tensor K_i^{-1} + 1 \tensor F_i, K_i \tensor K_i, \\
\epsilon : E_i, F_i, K_{h} &\mapsto 0, 0, 1 \\
S : E_i, F_i, K_{h} &\mapsto -K_i^{-1} E_i, -F_i K_i, K_{-h}.
\end{split}
\end{equation*}
With these actions $\Uq(\mathfrak{g})$ becomes a Hopf algebra.
Let $\Uq(\mathfrak{g}')$ be the Hopf subalgebra of $\Uq(\mathfrak{g})$ generated by $E_i$, $F_i$ and $K_i^{\pm 1}$ for all $i \in I$.

For Theorem \ref{q-radial_thm_quantumiwasawa} and Theorem \ref{q-radial_thm_main} we have to extend the quantized universal enveloping algebra with the roots of $K_{h}$.
We denote $\check{\mathcal{U}}_q(\mathfrak{g})$ for the associative $\CC(q)$-algebra generated by $\Uq(\mathfrak{g})$ and $K_{h}$ for $h \in \check{P}^{\vee}$ with the same relations (\ref{q-radial_eqn_uqgrelations}) taking $h, h' \in \check{P}^{\vee}$.
The Hopf subalgebra $\check{\mathcal{U}}_q(\mathfrak{g}')$ of $\check{\mathcal{U}}_q(\mathfrak{g})$ is generated by $E_i$, $F_i$ and $K_{h}$ for $h \in \check{Q}^{\vee}$.
Note that $\mathcal{U}_q(\mathfrak{g})$ is a Hopf subalgebra of $\check{\mathcal{U}}_q(\mathfrak{g})$ and that $\mathcal{U}_q(\mathfrak{g}')$ is a Hopf subalgebra of $\check{\mathcal{U}}_q(\mathfrak{g}')$.

Let $U^+$, $U^-$ and $U^0$ be the Hopf subalgebras of $\Uq(\mathfrak{g})$ generated respectively by $\{E_i : i \in I\}$, $\{F_i : i \in I\}$ and $\{K_h : h \in P^{\vee}\}$.
Take $\check{U}^0$ to be the Hopf subalgebra of $\check{\mathcal{U}}_q(\mathfrak{g})$ generated by $\{K_h : h \in \check{Q}^{\vee}\}$.
By \cite[\S 3.2]{Lus94} we have $U^+ \tensor U^0 \tensor U^{-} \simeq \Uq(\mathfrak{g})$ and $U^+ \tensor \check{U}^0 \tensor U^{-} \simeq \check{\mathcal{U}}_q(\mathfrak{g})$ as vector spaces under the multiplication map.
Write $U^{0 \prime}$ for the subalgebra of $U^0$ generated by all elements $\{K_i^{\pm 1} : i \in I\}$ and write $\check{U}^{0 \prime}$ for the subalgebra of $\check{U}^0$ generated by $\{ K_h : h \in \check{Q}^{\vee} \}$.
By \cite[\S 3.2]{Lus94} we have $U^+ \tensor U^{0 \prime} \tensor U^- \simeq \Uq(\mathfrak{g}')$ and $U^+ \tensor \check{U}^{0 \prime} \tensor U^- \simeq \check{\mathcal{U}}_q(\mathfrak{g}')$ as vector spaces under the multiplication map.
If $\rank(A) = |I|$ then $\Uq(\mathfrak{g}) = \Uq(\mathfrak{g}')$, and in explicit cases we write $\mathcal{U}_q(\mathfrak{g})$ for $\mathcal{U}_q(\mathfrak{g}')$.

For $\CC(q)[Q]$, the group algebra of the root lattice, we define an algebra isomorphism $\CC(q)[Q] \to U^{0\prime}$ defined on the generators by $\alpha_i \mapsto K_i$.
For any $\beta \in Q$ we write
\begin{equation} \label{q-radial_eqn_KinQ}
K_{\beta} = \prod_{i \in I} K_{i}^{n_i}, \quad \text{where } \beta = \sum_{i \in I} n_i \alpha_i.
\end{equation}
By \eqref{q-radial_eqn_KinQ} we find commutation relations of the form
\begin{equation*}
K_{\beta} E_i = q^{(\beta, \alpha_i)} E_i K_{\beta}, \quad K_{\beta} F_i = q^{-(\beta, \alpha_i)} F_i K_{\beta},
\end{equation*}
for all $\beta \in Q$ and $i \in I$.
By the same argument there exists an algebra isomorphism $\CC(q)[\check{Q}] \to \check{U}^{0\prime}$ and we define $K_{\beta} \in \check{U}^{0\prime}$ for any $\beta \in \check{Q}$ similar to \eqref{q-radial_eqn_KinQ}.

Take $n \in \NN$ and $U = (u_1, u_2, \ldots, u_n) \in I^n$, we abbreviate $F_U = F_{u_1} F_{u_2} \ldots F_{u_n}$ and $E_U = E_{u_1} E_{u_2} \ldots E_{u_n}$.

\section{Quantum symmetric pairs} \label{q-radial_sec_symmetricquantumkacmoodypairs}

We introduce admissible pairs which are a generalization of the Satake diagrams as given in \cite{Ara62}, see also Kolb \cite[Definition 2.3]{Kol14}.

Let $X \subseteq I$ such that $\mathfrak{g}_X$ is of finite type, see \cite[p. 399]{Kol14}. 
Write $W_X \subseteq W$ for the corresponding parabolic subgroup of $W$ with longest element $w_{X}$ and $\Phi_X \subseteq \Phi$ for the corresponding root system.
Let $\rho^{\vee}_X$ be the half sum of the positive coroots of $\Phi_X$.
For $\Aut(A)$ we denote the group of all permutations $\sigma$ on $I$ such that $a_{i,j} = a_{\sigma(i), \sigma(j)}$.
A pair $(X, \tau)$ where $X \subseteq I$ and $\tau \in \Aut(A, X) = \{ \sigma \in \Aut(A) : \sigma(X) = X \}$ is called an admissible pair if
\begin{enumerate}
\item $\tau^2 = \id_I$,
\item The action of $\tau$ on $X$ coincides with the action of $-w_X$,
\item If $j \in I \backslash X$ and $\tau(j) = j$, then $\alpha_j(\rho_X^{\vee}) \in \ZZ$.
\end{enumerate} 
Given an admissible pair $(X, \tau)$ we define an involution $\theta = \theta(X, \tau)$ by \cite[Theorem 2.5]{Kol14} such that on $\mathfrak{h}$ we have $\theta(h) = -w_X\tau(h)$.
By duality $\theta$ induces a map $\Theta : \mathfrak{h}^* \to \mathfrak{h}^* : \alpha \mapsto -w_X \tau(\alpha)$. 
Let $\mathfrak{k}' = \{ x \in \mathfrak{g}' : \theta(x) = x \}$ be the fixed point Lie subalgebra of $\mathfrak{g}'$ with respect to involution $\theta$.
Take $\mathcal{M}_X$ to be the subalgebra of $\Uq(\mathfrak{g}')$ generated by $\{E_i, F_i, K_i^{\pm 1} : i \in X\}$ and $\check{\mathcal{M}}_X$ to be the subalgebra of $\check{\mathcal{U}}_q(\mathfrak{g}')$ generated by $\{E_i, F_i : i \in X\}$ and $\{K_{h} : h \in \frac{1}{2}\mathbb{Z}[\{\alpha_i : i \in X\}]\}$.
The quantum torus $\mathcal{A} = A_{\Theta}$ is generated by all elements $K_{\alpha}$ where $\alpha \in Q$ such that $\Theta(\alpha) = -\alpha$ and write $\check{\mathcal{A}} = \check{A}_{\Theta}$ for the quantum torus generated by $K_{\alpha}$ where $\alpha \in \check{Q}$ and $\Theta(\alpha) = -\alpha$.
Let $U^{0\prime}_{\Theta}$ be the subalgebra of $U^{0\prime}$ consisting of all the fixed points $K_{\beta}$, where $\beta \in Q$ and $\Theta(\beta) = \beta$ and let $\check{U}^{0\prime}_{\Theta}$ be the subalgebra of $\check{U}^{0\prime}$ with elements $K_{\beta}$ such that $\beta \in \check{Q}$ and $\Theta(\beta) = \beta$.

For every admissible pair $(X, \tau)$, Kolb \cite[Definition 4.3]{Kol14} defines the quantum involution $\theta_q = \theta_q(X, \tau) : \Uq(\mathfrak{g}') \to \Uq(\mathfrak{g}')$.
In general the quantum involution $\theta_q$ is not a Hopf algebra automorphism and is not an involution, i.e. $\theta_q^2 \neq \id$, in general.
However we always have $\theta_q(K_h) = K_{\theta(h)}$, $\theta_q | \mathcal{M}_X = \id_{\mathcal{M}_X}$ and $\theta_q \to \theta$ for $q \rightarrow 1$.

\begin{definition} \label{q-radial_dfn:simplegenerator}
For an admissible pair $(X, \tau)$ take $\textbf{c} = (c_i)_{i \in I \backslash X} \in (\CC(q)^{\times})^{I \backslash X}$, $\textbf{s} = (s_i)_{i \in I \backslash X} \in (\CC(q))^{I \backslash X}$ and define $\mathcal{B} = \mathcal{B}_{\textbf{c}, \textbf{s}} = \mathcal{B}_{\textbf{c}, \textbf{s}}(X, \tau)$ to be the subalgebra of $\Uq(\mathfrak{g}')$ generated by $\mathcal{M}_X$, $U^{0 \prime}_{\Theta}$ and
\begin{equation*}
B_i^{\textbf{c}, \textbf{s}} := B_i := F_i + c_i \theta_q(F_i K_i) K_i^{-1} + s_i K_i^{-1}, \quad i \in I \backslash X.
\end{equation*}
The pair $(\Uq(\mathfrak{g}'), \mathcal{B})$ is called the quantum symmetric pair related to the admissible pair $(X, \tau)$.
If $s = \textbf{0}$ we often write $\mathcal{B}_{\textbf{c}} = \mathcal{B}_{\textbf{c}, \textbf{0}}$.
Note that $\mathcal{B}$ is a right coideal of $\Uq(\mathfrak{g}')$ \cite[Proposition 5.2]{Kol14}, i.e. $\Delta(\mathcal{B}) \subseteq \mathcal{B} \tensor \Uq(\mathfrak{g}')$.
If, for all $i \in I \backslash X$, we have $\theta_q(F_i K_i) = -v_i E_{\tau(i)}$, where $v_i \in \CC(q)^\times$, we call $(\Uq(\mathfrak{g}'), \mathcal{B})$ the quantum symmetric pair with simple generators (related to the admissible pair $(X, \tau)$).

The algebra $\check{\mathcal{B}}$ of algebra $\check{\Uq}(\mathfrak{g}')$ is generated by $\{B_i\}_{i \in I \backslash X}, \check{\mathcal{M}}_X$ and $\check{U}^{0 \prime}_{\Theta}$.
The quantum symmetric pair $(\check{\Uq}(\mathfrak{g}'), \check{\mathcal{B}})$ has simple generators if $(\Uq(\mathfrak{g}'), \mathcal{B})$ has simple generators.
\end{definition}

\begin{definition}
Let $\mathcal{B}_1$ and $\mathcal{B}_2$ be two right coideals of $\Uq(\mathfrak{g}')$.
We call $\mathcal{B}_1$ equivalent to $\mathcal{B}_2$ if there exists a Hopf algebra isomorphism $\phi$ on $\Uq(\mathfrak{g}')$ such that $\phi(\mathcal{B}_1) = \mathcal{B}_2$.
\end{definition}

We define the left adjoint action of $\Uq(\mathfrak{g})$ on itself by
\begin{equation*}
\ad(x)(y) = \sum_{(x)} x_{(1)} y S(x_{(2)}),
\end{equation*}
where $x, y \in \Uq(\mathfrak{g})$.

\begin{proposition} \label{q-radial_prop_simple_B}
Let $(X, \tau)$ be an admissible pair for the quantum symmetric pair $(\Uq(\mathfrak{g}'), \mathcal{B}_{\textbf{c}, \textbf{s}})$.
The quantum symmetric pair $(\Uq(\mathfrak{g}'), \mathcal{B}_{\textbf{c}, \textbf{s}})$ has simple generators if and only if for every $i \in I \backslash X$ the $\ad(\mathcal{M}_X)$-module $\ad(\mathcal{M}_X)(E_i)$ is one dimensional.
\end{proposition}

\begin{proof}
Suppose that $(\Uq(\mathfrak{g}'), \mathcal{B}_{\textbf{c}, \textbf{s}})$ has simple generators. 
Fix $i \in I \backslash X$.
There exists a highest weight vector of $\ad(\mathcal{M}_X)(E_i)$ of the form $\ad(Z_i^+(X))E_i)$, with $Z_i^+(X) = E_V$ and $V = (j_1, j_2, \ldots, j_r)$, see \cite[Lemma 3.5, (4.3) and (4.4)]{Kol14}.
Also by \cite[Lemma 3.5]{Kol14} $E_i$ is a lowest weight vector of $\ad(\mathcal{M}_X)(E_i)$.
Let $j \in I$ such that $\tau(j) = i$, then by \cite[Theorem 4.4.(3)]{Kol14} we have $\theta_q(X,\tau)(F_j K_j) = -w_i \ad(Z_{i}^+)(E_{i})$ for some $w_i \in \CC(q)^{\times}$.
Since $(\Uq(\mathfrak{g}'), \mathcal{B}_{\textbf{c}, \textbf{s}})$ has simple generators $\theta_q(X, \tau)(F_j K_j) = -v_i E_i$ for some $v_i \in \CC(q)^{\times}$.
Therefore $E_i$ is a highest weight vector and a lowest weight vector for the irreducible $\ad(\mathcal{M}_X)$-module $\ad(\mathcal{M}_X)(E_i)$.
This shows that $\ad(\mathcal{M}_X)(E_i)$ is one dimensional.

Suppose for all $i \in I \backslash X$ we have that the $\ad(\mathcal{M}_X)$-module $\ad(\mathcal{M}_X)(E_i)$ is one dimensional, then by \cite[Theorem 4.4.(3)]{Kol14} $\theta_q(X,\tau)(F_i K_i) = -v_i E_{\tau(i)}$, where $v_i \in \CC(q)$.
By Definition \ref{q-radial_dfn:simplegenerator} we find that $(\Uq(\mathfrak{g}'), \mathcal{B}_{\textbf{c}, \textbf{s}})$ has simple generators.
\end{proof}

\begin{lemma} \label{q-radial_lem_adEi_alphaij}
Take $i,j \in I$, then $\ad(E_i)(E_j) = 0$ if and only if $a_{ij} = (\alpha_i, \alpha_j) = 0$.
\end{lemma}

\begin{proof}
Since $\Delta(E_i) = E_i \tensor 1 + K_i \tensor E_i$ we find that $\ad(E_i)(E_j)$ is
\begin{equation} \label{q-radial_eqn_adEiEj}
\sum_{(E_i)} (E_i)_{(1)} E_j S((E_i)_{(2)})
  = E_i E_j + K_i E_j S(E_j)
  = E_i E_j - q^{-(\alpha_i, \alpha_j)} E_j E_i.
\end{equation}
If $(\alpha_i, \alpha_j) = 0$, then from the quantum Serre relations \eqref{q-radial_eqn_adEiEj} is zero.
Suppose \eqref{q-radial_eqn_adEiEj} is zero, then $E_i E_j = q^{-(\alpha_i, \alpha_j)} E_j E_i$.
This can only follow from the quantum Serre relations if $(\alpha_i, \alpha_j) = 0$
\end{proof}

Proposition \ref{q-radial_prop_simple_B} states that a quantum symmetric pair has simple generators if and only if for every $i \in I \backslash X$ the $\ad(\mathcal{M}_X)$-module $\ad(\mathcal{M}_X)(E_i)$ is one-dimensional.
Suppose $X \neq \emptyset$, then $\ad(\mathcal{M}_X)(E_i)$ is one-dimensional for all $i \in I \backslash X$ if and only if $\ad(E_j)(E_i) = 0$ for all $j \in X$.
By Lemma \ref{q-radial_lem_adEi_alphaij}, $\ad(E_j)(E_i) = 0$ if and only if $a_{ij} = 0$.
This observation gives the following Corollary.

\begin{corollary} \label{q-radial_cor_alternative_simple_B}
Let $(X, \tau)$ be an admissible pair for quantum symmetric pair $(\Uq(\mathfrak{g}'), \mathcal{B}_{\textbf{c}, \textbf{s}})$.
The quantum symmetric pair $(\Uq(\mathfrak{g}'), \mathcal{B})$ has simple generators if and only if for all $i \in X$ and $j \in I \backslash X$ we have $\alpha_{i,j} = (\alpha_i, \alpha_j) = 0$.
In particular, if $X = \emptyset$, then $(\Uq(\mathfrak{g}'), \mathcal{B}(\emptyset, \tau))$ has simple generators. 
\end{corollary}

Not all right coideals $\mathcal{B}_{\textbf{c}, \textbf{s}}$ are suitable quantum analogues for $\mathcal{U}(\mathfrak{k}')$.
Kolb \cite[Lemma 5.3, 5.4 \& 5.5]{Kol14} showed that for $\mathcal{B}_{\textbf{c}, \textbf{s}}$ to be a suitable quantum analogue for $\mathcal{U}(\mathfrak{k}')$ we need restrictions on $\textbf{c}$ and $\textbf{s}$.
Define
\begin{equation*}
I_{\textrm{ns}} = \{ i \in I \backslash X : \tau(i) = i, \alpha_i(h_j) = 0 \text{ for all } j \in X \},
\end{equation*}
then $\mathcal{B}_{\textbf{c}, \textbf{s}}$ is a suitable quantum analogue for $\mathcal{U}(\mathfrak{k}')$ if $\textbf{c} \in \mathcal{C}$ and $\textbf{s} \in \mathcal{S}$, where
\begin{equation*}
\begin{split}
\mathcal{C} &= \{ \textbf{c} \in (\CC(q)^{\times})^{|I \backslash X|} : c_i = c_{\tau(i)} \text{ if } \tau(i) \neq i \text{ and } (\alpha_i, \Theta(\alpha_i)) = 0 \}, \\
\mathcal{S} &= \{ \textbf{s} \in (\CC(q))^{|I \backslash X|} : \text{if } s_i \neq 0 \text{ then } i \in I_{\textrm{ns}} \text{ and } a_{ij} \in -2\NN_0 \text{ for all } j \in I_{\textrm{ns}} \backslash \{i\} \}.
\end{split}
\end{equation*}

The construction of the quantum symmetric pair subalgebra $\mathcal{B}_{\textbf{c}, \textbf{s}}$ seems to be very artificial.
However if $\textbf{c} \in \mathcal{C}$ and $\textbf{s} \in \mathcal{S}$, then $\mathcal{B}_{\textbf{c}, \textbf{s}}$ specializes to $\mathrm{U}(\mathfrak{k})$ for $q = 1$ and $\mathcal{B}_{\textbf{c}, \textbf{s}}$ is maximal with this property.
Letzter \cite[Theorem 5.8]{Let99}, \cite[Theorem 7.5]{Let03} showed that for finite dimensional $\mathfrak{g}$ any maximal coideal subalgebra of $\Uq(\mathfrak{g})$ that specializes to $\mathrm{U}(\mathfrak{k})$ for $q = 1$ is equivalent to a right coideal $\mathcal{B}_{\textbf{c}, \textbf{s}}$, see also Kolb \cite[Remark 5.7 and Section 10]{Kol14}.

For the rest of this section we assume $\textbf{c} \in \mathcal{C}$ and $\textbf{s} \in \mathcal{S}$.
We classify all quantum symmetric pairs with simple generators for semi-simple Lie algebras using the theory developed by Araki \cite{Ara62} and Letzter \cite{Let03}.
After the classification we work out a couple of these examples.
In the last example we work out the quantized universal enveloping algebra of affine $\hat{\mathfrak{sl}}_2$, which is infinite dimensional and hence not semi-simple.
For all the examples below we have $X = \emptyset$.

\begin{example}
Every symmetric pair $(G, K)$ of a compact Lie group $G$ gives rise to a symmetric pair $(\mathfrak{g}, \mathfrak{k})$ of semi-simple Lie algebras related to an admissible pair $(X, \tau)$.
Araki \cite{Ara62} gives the classification for all admissible pairs $(X, \tau)$ of symmetric compact finite dimensional Lie groups, see also \cite[Chapter X]{Hel78}.
Each of the admissible pairs $(X, \tau)$ of \cite{Ara62} gives rise to a quantum involution $\theta_q(X, \tau)$, see \cite[Definition 4.3]{Kol14}, from which we can construct a quantum symmetric pair $(\Uq(\mathfrak{g}), \mathcal{B})$.
By a case by case check on the list of Araki \cite{Ara62} we give all admissible pairs related to a symmetric pair $(G, K)$ of compact groups, using Corollary \ref{q-radial_cor_alternative_simple_B}, corresponding to a quantum symmetric pair with simple generators.
\begin{enumerate}
\item Let $\mathfrak{g}$ be a simple Lie algebra generated by $\{e_i, f_i, h_i : 1 \leq i \leq m\}$.
Take two copies $\mathfrak{g}_1$ and $\mathfrak{g}_2$ of $\mathfrak{g}$ and label the generators of $\mathfrak{g}_1$ by $\{e_i, f_i, h_i : 1 \leq i \leq m\}$ and label the generators of $\mathfrak{g}_2$ by $\{e_{i+m}, f_{i+m}, h_{i+m} : 1 \leq i \leq m\}$, take $\mathfrak{g} = \mathfrak{g}_1 \oplus \mathfrak{g}_2$.
The intertwiner on $\mathfrak{g}_1 \oplus \mathfrak{g}_2$ is given by $\theta : e_i, f_i, h_i \mapsto e_{i+m}, f_{i+m}, h_{i+m}$ and $\theta : e_{i+m}, f_{i+m}, h_{i+m} \mapsto e_i, f_i, h_i$, for $1 \leq i \leq m$.
Then $(\mathfrak{g}_1 \oplus \mathfrak{g}_2, \diag)$ correspond to the quantum symmetric pair $(\Uq(\mathfrak{g}_1 \oplus \mathfrak{g}_2) \simeq \Uq(\mathfrak{g}_1) \tensor \Uq(\mathfrak{g}_2), \mathcal{B}_{\textbf{c}, \textbf{s}})$.
Every right coideal for the quantum symmetric pair related to admissible pair $(X, \tau)$ is  equivalent to $\mathcal{B}_{\textbf{1}, \textbf{0}}$.
See Examples \ref{q-radial_exm_A1xA1} and \ref{q-radial_exm_A2xA2} for quantum symmetric pairs related to $(\SU(2) \times \SU(2), \diag)$ and $(\SU(3) \times \SU(3), \diag)$.
\item Type AI, where we have $I = \{1, 2, \ldots, r\}$ and $\tau = \id$.
In this case every right coideal $\mathcal{B}$ for quantum symmetric pair of type AI is equivalent to $\mathcal{B}_{\textbf{1}, \textbf{0}}$.
\item Type AIII, case $2$, where $I = \{1, 2, \ldots, r\}$ for odd $r = 2\ell + 1$, $\ell \in \NN$.
The permutation $\tau$ is defined by $i \mapsto r-i+1$.
Every right coideal $\mathcal{B}$ for quantum symmetric pair of type AIII is equivalent to $\mathcal{B}_{1,\textbf{s}}$, where $\textbf{s} = (0, 0, \ldots, 0, s, 0, \ldots, 0)$ with an $s \in \CC(q)$ on entry $\ell$ of $\textbf{s}$.

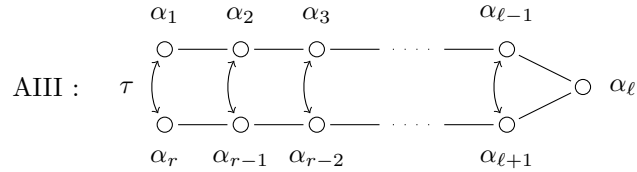
\begin{center}
\begin{tikzpicture}
\draw (-1, 0) node[anchor=east] {AIII :};

\foreach \x in {0,1,2} {
  \def \y {\number\numexpr\x+1\relax};
  \draw (\x,0.5) circle (.1cm) node[label={[label distance=.1cm]90:$\alpha_{\y}$}] {};
  \draw (\x,-0.5) circle (.1cm) node[label={[label distance=.1cm]-90:\ifthenelse{\x=0}{$\alpha_r$}{$\alpha_{r-\x}$}}] {};
  \draw[shorten >= 5pt, shorten <= 5pt] (\x,0.5) to (\x+1,0.5);
  \draw[shorten >= 5pt, shorten <= 5pt] (\x,-0.5) to (\x+1,-0.5);
  \draw[shorten >= 5pt, shorten <= 5pt, <->] (\x,0.5) to [bend right=30] (\x,-0.5);
}

\draw[loosely dotted] (3,0.5) to (3.5,0.5);
\draw[loosely dotted] (3,-0.5) to (3.5,-0.5);
  
\draw (4.5,0.5) circle (.1cm) node[label={[label distance=.1cm]90:$\alpha_{\ell-1}$}] {};
\draw (4.5,-0.5) circle (.1cm) node[label={[label distance=.1cm]-90:$\alpha_{\ell+1}$}] {};
\draw[shorten >= 5pt, shorten <= 5pt] (3.5,0.5) to (3.5+1,0.5);
\draw[shorten >= 5pt, shorten <= 5pt] (3.5,-0.5) to (3.5+1,-0.5);
\draw[shorten >= 5pt, shorten <= 5pt, <->] (4.5,0.5) to [bend right=30] (4.5,-0.5);

\draw (5.5,0) circle (.1cm) node[label={[label distance=.1cm]0:$\alpha_{\ell}$}] {};
\draw[shorten >= 5pt, shorten <= 5pt] (4.5,0.5) to (5.5,0);
\draw[shorten >= 5pt, shorten <= 5pt] (4.5,-0.5) to (5.5,0);

\draw (-0.5,0) node {$\tau$};

\end{tikzpicture}
\end{center}

\item Type AIV when $r=1,2$, i.e. the quantum symmetric pairs related to $(\SU(2), \mathrm{U}(1))$ and $(\SU(3), \mathrm{U}(2))$.

For $r = 1$ we have $I = \{1\}$ and $\tau = \id$.
The right coideal $\mathcal{B}$ for the quantum symmetric pair is equivalent to $\mathcal{B}_{1,s}$ for $s \in \CC(q)$.
See Example \ref{q-radial_exm_koornwinder1}.

For $r = 2$ we have $I = \{1, 2\}$ and $\tau = (1\, 2)$.
The right coideal $\mathcal{B}$ for the quantum symmetric pair is equivalent to $\mathcal{B}_{\textbf{c}, \textbf{0}}$ where $\textbf{c} = (c, c)$ for $c \in \CC(q)^{\times}$.
See also Example \ref{q-radial_exm_uqsl3}.
\item Type BI, for $\ell = r$, so that $I = \{1, 2, \ldots, r\}$ and $\tau = \id$.
Every right coideal for the quantum symmetric pair of type BI is equivalent to $\mathcal{B}_{\textbf{1}, \textbf{0}}$.
\item Type BII, for $r = 1$, so that $I = \{1\}$ and $\tau = \id$.
Every right coideal for the quantum symmetric pair of type BII with $r = 1$ is equivalent to $\mathcal{B}_{1,0}$.
\item Type CI, where $I = \{1, 2, \ldots, r\}$ and $\tau = \id$.
Every right coideal for the quantum symmetric pair of type CI is equivalent to $\mathcal{B}_{\textbf{1}, \textbf{s}}$ where $\textbf{s} = (0,0,\ldots,0,s)$ with $s \in \CC(q)$.
\item Type DI, case 2, where $r \geq 3$, so that $I = \{1, 2, \ldots, r\}$ and $\tau = ((r-1)\,r)$.
Every right coideal for the quantum symmetric pair of type DI, case 2, is equivalent to $\mathcal{B}_{\textbf{1}, \textbf{0}}$.

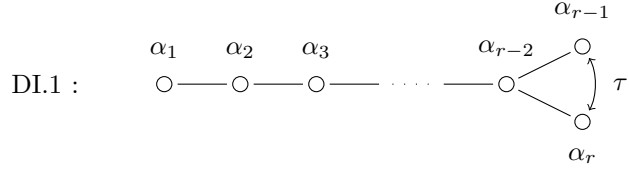
\begin{center}
\begin{tikzpicture}
\draw (-1, 0) node[anchor=east] {DI.1 :};

\foreach \x in {0,1,2} {
  \def \y {\number\numexpr\x+1\relax};
  \draw (\x,0) circle (.1cm) node[label={[label distance=.1cm]90:$\alpha_{\y}$}] {};
  \draw[shorten >= 5pt, shorten <= 5pt] (\x,0) to (\x+1,0);
}

\draw[loosely dotted] (3,0) to (3.5,0);

\draw[shorten >= 5pt, shorten <= 5pt] (3.5,0) to (4.5,0);
\draw (4.5,0) circle (.1cm) node[label={[label distance=.1cm]90:$\alpha_{r-2}$}] {};

\draw (5.5,0.5) circle (.1cm) node[label={[label distance=.1cm]90:$\alpha_{r-1}$}] {};
\draw[shorten >= 5pt, shorten <= 5pt] (4.5,0) to (5.5,0.5);

\draw (5.5,-0.5) circle (.1cm) node[label={[label distance=.1cm]-90:$\alpha_{r}$}] {};
\draw[shorten >= 5pt, shorten <= 5pt] (4.5,0) to (5.5,-0.5);

\draw[shorten >= 5pt, shorten <= 5pt, <->] (5.5,0.5) to [bend left=30] (5.5,-0.5);

\draw (6,0) node {$\tau$};
\end{tikzpicture}
\end{center}

\item Type DI, case 3, where $r \geq 4$, so that $I = \{1, 2, \ldots, r\}$ and $\tau = \id$.
Every right coideal for the quantum symmetric pair of type DI, case 3, is equivalent to $\mathcal{B}_{\textbf{1}, \textbf{0}}$.
\item Type EI, where $I = \{1, 2, \ldots, 6\}$ and $\tau = \id$.
Every right coideal for the quantum symmetric pair of type EI is equivalent to $\mathcal{B}_{\textbf{1}, \textbf{0}}$.
\item Type EII, where $I = \{1, 2, \ldots, 6\}$ and $\tau = (1\,6)(3\,5)$.
Every right coideal for the quantum symmetric pair of type EII is equivalent to $\mathcal{B}_{\textbf{1}, \textbf{0}}$.

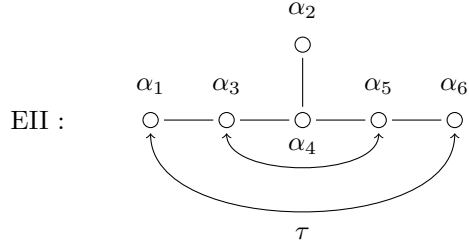
\begin{center}
\begin{tikzpicture}
\draw (-1, 0) node[anchor=east] {EII :};

\draw (0,0) circle (.1cm) node[label={[label distance=.1cm]90:$\alpha_{1}$}] {};
\draw [shorten >= 5pt, shorten <= 5pt] (0,0) to (1,0);
\draw (1,0) circle (.1cm) node[label={[label distance=.1cm]90:$\alpha_{3}$}] {};
\draw [shorten >= 5pt, shorten <= 5pt] (1,0) to (2,0);
\draw (2,0) circle (.1cm) node[label={[label distance=0cm]-90:$\alpha_{4}$}] {};
\draw [shorten >= 5pt, shorten <= 5pt] (2,0) to (3,0);
\draw [shorten >= 5pt, shorten <= 5pt] (2,0) to (2,1);
\draw (3,0) circle (.1cm) node[label={[label distance=.1cm]90:$\alpha_{5}$}] {};
\draw [shorten >= 5pt, shorten <= 5pt] (3,0) to (4,0);
\draw (4,0) circle (.1cm) node[label={[label distance=.1cm]90:$\alpha_{6}$}] {};
\draw (2,1) circle (.1cm) node[label={[label distance=.1cm]90:$\alpha_{2}$}] {};

\draw[shorten >= 5pt, shorten <= 5pt, <->] (1,0) to [bend right=90] (3,0);
\draw[shorten >= 5pt, shorten <= 5pt, <->] (0,0) to [bend right=90] (4,0);
\draw (2,-1.5) node {$\tau$};

\end{tikzpicture}
\end{center}

\item Type EV, where $I = \{1, 2, \ldots, 7\}$ and $\tau = \id$.
Every right coideal for the quantum symmetric pair of type EV is equivalent to $\mathcal{B}_{\textbf{1}, \textbf{0}}$.
\item Type EVIII, where $I = \{1, 2, \ldots, 8\}$ and $\tau = \id$.
Every right coideal for the quantum symmetric pair of type EVIII is equivalent to $\mathcal{B}_{\textbf{1}, \textbf{0}}$.
\item Type FI, where $I = \{1, 2, 3, 4\}$ and $\tau = \id$.
Every right coideal for the quantum symmetric pair of type FI is equivalent to $\mathcal{B}_{\textbf{1}, \textbf{0}}$.
\item Type G, where $I = \{1, 2\}$ and $\tau = \id$.
Every right coideal for the quantum symmetric pair of type G is equivalent to $\mathcal{B}_{\textbf{1}, \textbf{0}}$.
\end{enumerate}
\end{example}

\begin{example}[Type $A_1$] \label{q-radial_exm_koornwinder1}
Consider quantum symmetric pairs $(\Uq(\mathfrak{sl}_2), \mathcal{B}_{c,s})$ for type AIV of Araki \cite{Ara62}.
$\Uq(\mathfrak{sl}_2)$ is generated by $E, F$ and $K^{\pm 1}$.
The admissible pair is $(\emptyset, \id)$ and the right coideal $\mathcal{B}_{1,s}$, where $s \in \CC(q)$, is generated by
\begin{equation*}
B = F - E K^{-1} + s K^{-1}.
\end{equation*}
Therefore the quantum pair $(\mathcal{U}_q(\mathfrak{sl}_2), \mathcal{B}_{1,s})$ corresponding to Gelfand pair $(\SU(2), \mathrm{U}(1))$ has simple generators.
The quantum torus $\mathcal{A}$ is generated by $K^{\pm 1}$ and therefore we call $(\Uq(\mathfrak{sl}_2), \mathcal{B}_{1,s})$ a pair of rank $1$.
Fix $c \in \CC(q)^{\times}$ and $s \in \CC(q)$ and take Hopf algebra isomorphism $\phi : E, F, K_{\mu} \mapsto c^{-\frac{1}{2}} E, c^{\frac{1}{2}} F, K_{\mu}$.
We see that $\phi(\mathcal{B}_{c,s}) = \mathcal{B}_{1, c^{\frac{1}{2}} s}$, hence the right coideal $\mathcal{B}_{c,s}$ is equivalent to the right coideal $\mathcal{B}_{1,t}$ for some $t \in \CC(q)$.
This case has been studied first by Koornwinder \cite{Koo93}. 
\end{example}

\begin{example}[Type $A_1 \times A_1$] \label{q-radial_exm_A1xA1}
Let $\mathfrak{g} = \mathfrak{sl}_2 \oplus \mathfrak{sl}_2$.
Consider the quantum symmetric pairs $(\Uq(\mathfrak{g}), \mathcal{B})$ related to $(\SU(2)\times\SU(2), \diag)$. 
The admissible pair is given by $(\emptyset, (1\,2))$.
A quick computation gives $\mathcal{C} = \{(c,c) : c \in \CC(q)^{\times} \}$ and $\mathcal{S} = \{ (0,0) \}$.
For $\textbf{c} = (1,1)$ and $\textbf{s} = 0$, the right coideal $\mathcal{B}$ is generated by
\begin{equation*}
B_1 = F_1 - E_2 K_1^{-1}, \quad B_2 = F_2 - E_1 K_2^{-1}, \quad K^{\pm 1} = (K_1 K_2^{-1})^{\pm 1}.
\end{equation*}
Therefore $(\Uq(\mathfrak{g}), \mathcal{B})$ is a quantum symmetric pair with simple generators.
The quantum torus $\mathcal{A}$ is generated by $(K_1 K_2)^{\pm 1}$ and hence we call $(\Uq(\mathfrak{g}), \mathcal{B})$ a pair of rank $1$.
Let $\textbf{c} \in \mathcal{C}$ be arbitrary and define Hopf algebra isomorphism $\phi : E_i, F_i, K_{\mu} \mapsto c^{-\frac{1}{2}} E_i, c^{\frac{1}{2}} F_i, K_{\mu}$.
We see that $\phi(\mathcal{B}_{\textbf{c}, (0,0)}) = \mathcal{B}_{(1,1), (0,0)}$ and therefore the right coideal $\mathcal{B}_{\textbf{c}, (0,0)}$ is equivalent to the right coideal $\mathcal{B}_{(1,1), (0,0)}$.
Aldenhoven, Koelink and Rom\'an study this case in \cite{AKR15}.
\end{example}

\begin{example}[Type $A_2$] \label{q-radial_exm_uqsl3}
Let $\mathfrak{g} = \mathfrak{sl}_3$.
Take the quantum symmetric pair $(\Uq(\mathfrak{g}), \mathcal{B})$ related to $(\SU(3), \mathrm{U}(2))$.
The quantized universal enveloping algebra $\Uq(\mathfrak{g})$ is generated by $E_1, E_2, F_1, F_2, K_1$ and $K_2$.
The admissible pair is $(\emptyset, (1\,2))$ and $\mathcal{C} = \CC(q)^{\times} \times \CC(q)^{\times}$, $\mathcal{S} = \{(0,0)\}$.
Let $c_1, c_2 \in \CC(q)^{\times}$ and take right coideal $\mathcal{B}_{(c_1, c_2)}$ generated by
\begin{equation*}
B_1 = F_1 - c_1 E_1 K_1^{-1}, \quad B_2 = F_2 - c_2 E_2 K_1^{-1}, \quad K^{\pm 1} = (K_1 K_2^{-1})^{\pm 1}.
\end{equation*}
Therefore $(\Uq(\mathfrak{g}), \mathcal{B}_{(c_1, c_2)})$ is a quantum symmetric pair with simple generators.
The quantum torus $\mathcal{A}$ is generated by $(K_1 K_2)^{\pm 1}$, hence we call $(\Uq(\mathfrak{g}), \mathcal{B}_{(c_1, c_2)})$ a pair of rank $1$.
Let $\phi$ be the Hopf algebra isomorphism defined on the generators by $\phi(E_1) = c_1^{-1} E_1$, $\phi(F_1) = c_1 F_1$ and as the identity on the other generators.
We see $\phi(\mathcal{B}_{(c_1, c_2)}) = \mathcal{B}_{(1, c_1^{-1} c_2)}$, hence every right coideal $\mathcal{B}_{(c_1, c_2)}$ is equivalent to a right coideal $\mathcal{B}_{(1, d)}$ for $d \in \CC(q)$.
So the quantum symmetric pair $(\Uq(\mathfrak{g}), \mathcal{B})$ related to $(\SU(3), \mathrm{U}(2))$ has essentially one free parameter.
\end{example}

\begin{example}[Type $A_2 \times A_2$] \label{q-radial_exm_A2xA2}
Let $\mathfrak{g} = \mathfrak{sl}_3 \oplus \mathfrak{sl}_3$.
Consider the quantum symmetric pair $(\Uq(\mathfrak{g}), \mathcal{B})$ related to $(\SU(3)\times\SU(3), \diag)$. 
The admissible pair is given by $(\emptyset, \tau)$ where $\tau = (1\,3)(2\,4)$.
We have $\mathcal{C} = \{(c_1, c_2, c_1, c_2) : c_1, c_2 \in \CC(q)^{\times} \}$ and $\mathcal{S} = \{ (0,0,0,0) \}$.
Therefore $\Theta(\emptyset, \tau)$ maps $\alpha_i \mapsto -\alpha_{\tau(i)}$ and for $\textbf{c} = \textbf{1} = (c_1,c_2,c_1,c_2) \in \mathcal{C}$ and $\textbf{s} = \textbf{0} = (0,0,0,0)$, the right coideal $\mathcal{B}_{\textbf{c}, \textbf{0}}$ is generated by
\begin{equation*}
B_1 = F_1 - c_1 E_3 K_1^{-1}, \quad B_2 = F_2 - c_2 E_4 K_2^{-1}, \quad B_3 = F_3 - c_1 E_1 K_3^{-1}, \quad B_4 = F_4 - c_2 E_2 K_4^{-1},
\end{equation*}
$(K_1 K_3^{-1})^{\pm 1}$ and $(K_2 K_4^{-1})^{\pm 1}$. 
Therefore $(\Uq(\mathfrak{sl}_3) \tensor \Uq(\mathfrak{sl}_3), \mathcal{B}_{\textbf{c},\textbf{0}})$ is a quantum symmetric pair with simple generators.
The quantum torus $\mathcal{A}$ is generated by $(K_1 K_3)^{\pm 1}$ and $(K_2 K_4)^{\pm 1}$, hence we call $(\Uq(\mathfrak{g}), \mathcal{B}_{\textbf{c},\textbf{0}})$ a pair of rank $2$.
Define the Hopf algebra isomorphism $\phi$ defined on the generators by $\phi(E_i) = c_1^{-\frac{1}{2}} E_i$, $\phi(F_i) = c_1^{\frac{1}{2}} F_i$ if $i = 1, 3$ and $\phi(E_i) = c_2^{-\frac{1}{2}} E_i$, $\phi(F_i) = c_2^{\frac{1}{2}} F_i$ for $i = 2, 4$.
We have $\phi(\mathcal{B}_{\textbf{c}, \textbf{0}}) = \mathcal{B}_{\textbf{1}, \textbf{0}}$, where $\textbf{1} = (1,1,1,1)$.
Hence the right coideal $\mathcal{B}_{\textbf{c}, \textbf{0}}$ with $\mathbf{c} \in \mathcal{C}$ is equivalent to $\mathcal{B}_{\textbf{1}, \textbf{0}}$.
\end{example}

\begin{example}[Quantum affine $\widehat{\mathfrak{sl}}_2$ and the $q$-Onsager algebra]
An example of a quantum symmetric pair not related to a symmetric pair of semi-simple Lie algebras is the quantum affine $\widehat{\mathfrak{sl}}_2$ and its $q$-Onsager subalgebra \cite{BB10, BK05}.
The embedding of the $q$-Onsager algebra in quantum affine $\widehat{\mathfrak{sl}}_2$ can be found in \cite[Proposition 1.13]{IT10}, see also \cite[\S 2]{Kol14}. 
The generalized Cartan matrix of affine Lie algebra  $\widehat{\mathfrak{sl}}_2(\CC)$ is given by $A = \begin{pmatrix} 2 & -2 \\ -2 & 2 \end{pmatrix}$ with $I = \{0, 1\}$.
The algebra $\Uq(\widehat{\mathfrak{sl}}_2)$ is called quantum affine $\widehat{\mathfrak{sl}}_2$.
We take the admissible pair $(X, \tau) = (\emptyset, \id)$, then $\mathcal{C} = \CC(q)^{\times} \times \CC(q)^{\times}$ and $\mathcal{S} = \{ (0,0) \}$
For any $\textbf{c} = (c_1, c_2) \in \mathcal{C}$, $\textbf{s} = (0,0)$ the right coideal $\mathcal{B}_{\textbf{c}, \textbf{s}}$ is generated by the simple elements
\begin{equation*}
B_1 = F_1 - c_1 E_1 K_1^{-1}, \quad B_2 = F_2 - c_2 E_2 K_2^{-1},
\end{equation*}
for $i \in I$.
The pair $(\Uq(\widehat{\mathfrak{sl}}_2), \mathcal{B}_{\textbf{c}, \textbf{s}})$ is a quantum symmetric pair with simple generators, see also \cite[Example 7.6]{Kol14}.
The quantum torus $\mathcal{A}$ is generated by $K_1^{\pm 1}$ and $K_2^{\pm 1}$ so that we call $(\Uq(\widehat{\mathfrak{sl}}_2), \mathcal{B}_{\textbf{c}, \textbf{s}})$ of rank $2$.
The right coideal $\mathcal{B}_{\textbf{c}, \textbf{s}}$ is isomorphic with the $q$-Onsager algebra.
In \cite{Kol13} Kolb studies the radial part of the Casimir of the Onsager algebra $\widehat{\mathfrak{sl}}_2$ in the fashion of Casselman and Mili\v{c}i{\'c} \cite{CM82}.
One can wonder whether the results achieved in \cite{Kol13} can be extended to quantum affine $\hat{\mathfrak{sl}}_2$ using Theorem \ref{q-radial_thm_main}.
\end{example}

\section{Quantum infinitesimal Cartan decomposition} \label{q-radial_sec_quantuminfcartandecomp}

In this section and Section \ref{q-radial_sec_sphericalfunctions} we assume that $\textbf{c}$ and $\textbf{s}$ are arbitrary finite sequences of elements of $\CC(q)$.
We do not assume in general $\textbf{c} \in \mathcal{C}$ or $\textbf{s} \in \mathcal{S}$.

\begin{lemma} \label{q-radial_lem_KinBA}
Let $(\mathcal{U}_q(\mathfrak{g}'), \mathcal{B}_{\textbf{c},\textbf{s}}$) be a quantum symmetric pair.
If $\alpha \in Q$ and $K \in \check{\mathcal{A}} = \check{A}_{\Theta}$ an element of the quantum torus, then $K_{\alpha} K \in \check{\mathcal{B}}_{\textbf{c}, \textbf{s}} \check{\mathcal{A}}$
\end{lemma}

\begin{proof}
We only have to check the statement on the generators of $\check{\mathcal{A}}$, therefore we assume $K = K_{\beta}$ for some $\beta \in \check{Q}$, with $\Theta(\beta) = -\beta$.
Moreover $\Theta$ is an involution on $\mathfrak{h}^*$.
Hence take $\gamma = \frac{1}{2}(\alpha + \Theta(\alpha))$ and $\delta = \frac{1}{2}(\alpha - \Theta(\alpha)) + \beta$, so that $\Theta(\gamma) = \gamma$ and $\Theta(\delta) = -\delta$.
Then by the construction of $\gamma$ and $\delta$ we have $K_{\gamma} \in \check{\mathcal{B}}_{\textbf{c},\textbf{s}}$ and $K_{\delta} \in \check{\mathcal{A}}$.
Therefore we have $K_{\alpha} K_{\beta} = K_{\alpha + \beta} = K_{\gamma + \delta} = K_{\gamma} K_{\delta} \in \check{\mathcal{B}}_{\textbf{c},\textbf{s}} \check{\mathcal{A}}$.
\end{proof}

\begin{example}
Take Example \ref{q-radial_exm_A1xA1}. 
We have $\Theta : \alpha_1, \alpha_2 \mapsto -\alpha_2, -\alpha_1$.
Furthermore the quantum torus $\check{\mathcal{A}}$ is generated by $K_{\lambda(\alpha_1 + \alpha_2)}$ with $\lambda \in \frac{1}{2}\ZZ$.
Let $\mu \in Q$, we decompose $K_{\mu}K_{\lambda(\alpha_1 + \alpha_2)}$ in $\check{\mathcal{B}}\check{\mathcal{A}}$.
Write $\mu = \mu_1 \alpha_1 + \mu_2 \alpha_2$.
According to the proof of Lemma \ref{q-radial_lem_KinBA} we have $\gamma = \frac{1}{2}(\mu + \Theta(\mu)) = \frac{1}{2}(\mu_1 - \mu_2) \alpha_1 + \frac{1}{2}(\mu_2 - \mu_1)\alpha_2$ and $\delta = \frac{1}{2}(\mu - \Theta(\mu)) + \lambda(\alpha_1 + \alpha_2) = (\frac{1}{2}(\mu_1 + \mu_2) + \lambda) (\alpha_1 + \alpha_2)$.
Hence we have $K_{\mu} K_{\lambda(\alpha_1 + \alpha_2)} = K_{\gamma} K_{\delta} \in \check{\mathcal{B}} \check{\mathcal{A}}$.
\end{example}

We now give a quantum analogue of the Iwasawa decomposition.
There are multiple variants of the quantum Iwasawa decomposition already known, see \cite[Theorem 2.2]{Let04}.
However, Letzter \cite{Let04} restricts to only finite dimensional Lie algebras.
The quantum Iwasawa decomposition given by Letzter \cite{Let04} decomposes $\Uq(\mathfrak{g})$ in the right coideal $\mathcal{B}$, the quantum torus $\mathcal{A}$ and the ``positive part'' $\mathcal{N}^+$ generated by $\textrm{Ad}(\mathcal{M}_X)(E_i)$ for all $i \in I \backslash X$.
However if the right coideal $\mathcal{B}$ has simple generators we show in Theorem \ref{q-radial_thm_quantumiwasawa} that $\mathcal{N}^+$ can be replaced by its counterpart generated by all $F_i$ for $i \in I \backslash X$.

\begin{theorem}[Quantum Iwasawa decomposition] \label{q-radial_thm_quantumiwasawa}
Let $(\mathcal{U}_q(\mathfrak{g}'), \mathcal{B} = \mathcal{B}_{\textbf{c}, \textbf{s}})$ be a quantum symmetric pair with simple generators.
Let $\check{\mathcal{A}} = \check{A}_{\Theta}$ be the quantum torus of $(\check{\mathcal{U}}_q(\mathfrak{g}'), \check{\mathcal{B}})$ and $\mathcal{N}$ be the subalgebra of $\mathcal{U}_q(\mathfrak{g}')$ generated by $F_i$ for $i \in I \backslash X$.
We have the quantum Iwasawa decomposition $\check{\mathcal{U}}_q(\mathfrak{g}') = \check{\mathcal{B}} \check{\mathcal{A}} \mathcal{N}$.
\end{theorem}

\begin{proof}
This proof is based on Kolb \cite[Proposition 6.1 and Proposition 6.3]{Kol14}.

Take integers $M, N \geq 0$, $U = (u_1, u_2, \ldots, u_M) \in I^M$, $V = (v_1, v_2, \ldots, v_N) \in I^N$ and $\beta \in \check{Q}$.
Because $U^+ \tensor \check{U}^0 \tensor U^- \simeq \check{\Uq}(\mathfrak{g}')$ as vector spaces under the multiplication map it is sufficient to show $E_U K_{\beta} F_V \in \check{\mathcal{B}} \check{\mathcal{A}} \mathcal{N}$.
We prove, with induction on $M$, that $E_U K_{\beta} F_V \in \check{\mathcal{B}} \check{\mathcal{A}} \mathcal{N}$.

By Corollary \ref{q-radial_cor_alternative_simple_B} and the Serre relations an element $E_i$, where $i \in X$, commutes with every $E_j$, where $j \in I \backslash X$.
Since $E_i \in \check{\mathcal{B}}_{\textbf{c}, \textbf{s}}$ we can pull all $E_i$, where $i \in X$, to the left and we can assume that $U \in (I \backslash X)^M$.
Similarly, an element $F_i$, where $i \in X$, commutes with every $E_j$ and $F_j$, where $j \in I \backslash X$.
Because $F_i$, where $i \in X$, $q$-commutes with every element in $\check{U}^0$ and $F_i \in \check{\mathcal{B}}_{\textbf{c}, \textbf{s}}$, we can pull $F_i$ to the left and we can assume $V \in (I \backslash X)^N$.

If $M = 0$ we have by Lemma \ref{q-radial_lem_KinBA} that $K_{\beta} F_V \in \check{\mathcal{B}} \check{\mathcal{A}} \mathcal{N}$.
Let $M > 0$ and assume that for all $M' < M$ and $U' \in (I \backslash X)^{M'}$ we have $E_{U'} K_{\beta} F_V \in \check{\mathcal{B}} \check{\mathcal{A}} \mathcal{N}$.
Since $u_1 \in I \backslash X$ there is $t \in I \backslash X$ such that $\tau(t) = u$ and $B_t = F_t - c_t E_{u_1} K_{t}^{-1} + s_t K_{t}^{-1}$, so
\begin{equation*}
E_{u_1} = \frac{1}{c_t} F_t K_t - B_t K_t + \frac{s_t}{c_t}.
\end{equation*}
Hence
\begin{equation*}
E_U K_{\beta} F_{V} = \frac{1}{c_t} F_t K_t E_{U'} K_{\beta} F_V - B_t K_t E_{U'} K_{\beta} F_V + \frac{s_t}{c_t} E_{U'} K_{\beta} F_V.
\end{equation*}
Directly from the induction hypothesis we see $E_{U'} K_{\beta} F_V \in \check{\mathcal{B}} \check{\mathcal{A}} \mathcal{N}$.
Since $K_t$ and $E_{U'}$ $q$-commute there is $x \in \mathbb{Q}$ such that $K_t E_{U'} = q^x E_{U'} K_t$ and
\begin{equation*}
B_t K_t E_{U'} K_{\beta} F_V = q^x B_{t} E_{U'} K_{\beta + \alpha_t} F_V \in \check{\mathcal{B}} \check{\mathcal{A}} \mathcal{N}.
\end{equation*}
Lastly for $F_t K_t E_{U'} K_{\beta} F_V = q^x F_t E_{U'} K_{\beta + \alpha_t} F_V$ we use the relation $[E_u, F_t] = \delta_{ut} (K_u - K_u^{-1})(q_u - q_u^{-1})^{-1}$ repeatedly so that we can pull $F_t$ through $E_{U'}$.
Let $i$ be the smallest integer such that $u_i = t$.
We have
\begin{equation}
\label{q-radial_eqn_Ftpull}
\begin{split}
F_t E_{U'} K_{\beta + \alpha_t} F_V 
  &= E_{u_1} \ldots E_{u_{i-1}} F_t E_{u_i} E_{u_{i+1}} \ldots E_{u_{M-1}} K_{\beta + \alpha_t} F_V \\
  &= -E_{u_1} \ldots E_{u_{i}} F_t E_{u_{i+1}} \ldots E_{u_{M-1}} K_{\beta + \alpha_t} F_V \\
    &\qquad + E_{u_1} \ldots E_{u_{i-1}} \left( \frac{K_{u_i} - K_{u_i}^{-1}}{q - q^{-1}} \right) E_{u_{i+1}} \ldots E_{u_{M-1}} K_{\beta + \alpha_t} F_V.
\end{split}
\end{equation}
By the induction hypothesis
\begin{equation*}
E_{u_1} \ldots E_{u_{i-1}} \left( \frac{K_{u_i} - K_{u_i}^{-1}}{q - q^{-1}} \right) E_{u_{i+1}} \ldots E_{u_{M-1}} K_{\beta + \alpha_t} F_V
\end{equation*}
is in $\check{\mathcal{B}} \check{\mathcal{A}} \mathcal{N}$ after moving the $K_{u_i}^{\pm 1}$ to the right at the cost of a power of $q$.
Therefore, after repeated application of \eqref{q-radial_eqn_Ftpull}, it is sufficient to show that $E_{U''} F_t K_{\beta + \alpha_t} F_V \in \check{\mathcal{B}} \check{\mathcal{A}} \mathcal{N}$ for all $U''$ such that $|U''| \leq |U'|$.
But $F_t$ $q$-commute with $K_{\beta + \alpha_t}$ hence
\begin{equation*}
E_{U''} F_t K_{\beta + \alpha_t} F_V = q^{(\alpha_t, \beta + \alpha_t)} E_{U''} K_{\beta + \alpha_t} F_t F_V = q^{(\alpha_t, \beta + \alpha_t)} E_{U''} K_{\beta + \alpha_t} F_{V'},
\end{equation*}
where $V' = (t, v_1, v_2, \ldots, v_N)$.
Now with the induction hypothesis $E_{U''} K_{\beta + \alpha_t} F_{V'} \in \check{\mathcal{B}} \check{\mathcal{A}} \mathcal{N}$, which yields the result.
\end{proof}

Let $K_{\lambda} \in \check{A}$ and $Y \in \check{\mathcal{U}}_q(\mathfrak{g}')$.
If we have two right coideals $\check{\mathcal{B}}_{\textbf{c},\textbf{s}}$ and $\check{\mathcal{B}}_{\textbf{d},\textbf{t}}$ with simple generators for $\check{\mathcal{U}}_q(\mathfrak{g}')$ Theorem \ref{q-radial_thm_quantumiwasawa} shows that we can write $K_{\lambda} Y \in \check{\mathcal{B}}_{\textbf{c},\textbf{s}} \check{\mathcal{A}} \mathcal{N}$.
The question that remains is if an element of $\check{\mathcal{A}} \mathcal{N}$ has a decomposition in $\check{\mathcal{B}}_{\textbf{c},\textbf{s}} \check{\mathcal{A}} \check{\mathcal{B}}_{\textbf{d},\textbf{t}}$.
To answer this question we first introduce for each element of $\Uq(\mathfrak{g}')$ the set of regular points of $\check{\mathcal{A}}$ of this element.

We introduce the notation $\qexp(x) = q^x$.

\begin{definition}
The set of regular points of $1$ is $\mathcal{A}_{\text{reg}}(1) = \check{\mathcal{A}}$.
Let $U = (u_1, u_2, \ldots, u_k) \in (I \backslash X)^k$, with $k > 0$, and define the set of regular points $\mathcal{A}_{\text{reg}}(F_U)$ to be all $K_{\lambda} \in \check{\mathcal{A}}$ such that
\begin{equation}
\label{q-radial_eqn_lem_usefull}
\frac{
  c_{v_1} c_{v_2} \ldots c_{v_{k - \ell}}
}{
  d_{v_1} d_{v_2} \ldots d_{v_{k - \ell}}
}
\neq
\qexp \left( \sum_{i = 1}^{k-\ell} (\lambda + \mu, \alpha_{v_i} + \alpha_{\tau(v_i)})
  - \sum_{\substack{1 \leq i, j < k - \ell, \\ i \neq j}} (\alpha_{v_i}, \alpha_{v_j}) \right),
\end{equation}
for all $0 \leq \ell \leq k$, where $(v_1, v_2, \ldots, v_{k-\ell})$ is any subsequence of $U$ and $\mu$ runs over the set
\begin{equation} \label{q-radial_eqn_domainmu}
\{N_1 \mu_{u_1} + N_2 \mu_{u_2} + \ldots + N_k \mu_{u_k} : |N_j| \leq 1, \text{ for all } 1 \leq j \leq k \},
\end{equation}
where $\mu_i = \frac{1}{2}(\alpha_i - \Theta(\alpha_i))$.

Let $Y \in \Uq(\mathfrak{g})$ and $K_{\lambda} \in \mathcal{A}$, then by the quantum Iwasawa decomposition, Theorem \ref{q-radial_thm_quantumiwasawa}, we can write $K_{\lambda} Y = \sum_{j} C_j(\lambda;Y) K_{\lambda_j(\lambda;Y)} F_{V_j(\lambda;Y)}$, where $C_j(\lambda;Y) \in \mathcal{B}_{\textbf{c},\textbf{s}}$, $\lambda_j(\lambda;Y) \in Q$ such that $\Theta(\lambda_j(\lambda;Y)) = -\lambda_j(\lambda;Y)$ and $F_{V_{j(\lambda;Y)}} \in \mathcal{N}$.
The set of regular points $\mathcal{A}_{\text{reg}}(Y)$ of $Y$ are all $\lambda \in Q$ such that $\lambda_j(\lambda;Y) \in \mathcal{A}_{\text{reg}}(F_{V_j(\lambda;Y)})$ for all $j$.
\end{definition}

\begin{remark} \label{q-radial_rmk_subseq}
Let $U = (u_1, u_2, \ldots, u_k) \in (I \backslash X)^k$ and assume $V = (v_1, v_2, \ldots, v_{k-\ell})$ is a subsequence of $U$, for $0 \leq \ell \leq k$.
Then $\mathcal{A}_{\textrm{reg}}(F_U) \subseteq \mathcal{A}_{\textrm{reg}}(F_V)$.
\end{remark}

In Theorem \ref{q-radial_thm_main}, which is the main theorem of this paper, we give sufficient conditions on the elements of $\mathcal{A} \mathcal{N}$ to decompose in $\check{\mathcal{B}}_{\textbf{c},\textbf{s}} \check{\mathcal{A}} \check{\mathcal{B}}_{\textbf{d},\textbf{t}}$.

\begin{theorem} \label{q-radial_thm_main}
Let $(X, \tau)$ be an admissible pair and let $\mathcal{B}_{\textbf{c},\textbf{s}} = \mathcal{B}_{\textbf{c}, \textbf{s}}(X,\tau)$ and $\mathcal{B}_{\textbf{d},\textbf{t}} = \mathcal{B}_{\textbf{d},\textbf{t}}(X,\tau)$ be two right coideals such that $(\mathcal{U}_q(\mathfrak{g}'), \mathcal{B}_{\textbf{c},\textbf{s}}$) and $(\mathcal{U}_q(\mathfrak{g}'), \mathcal{B}_{\textbf{d},\textbf{t}})$ are quantum symmetric pairs with simple generators.
Fix $U = (u_1, u_2, \ldots, u_k)$ with $u_j \in I \backslash X$ and let $K_{\lambda} \in \mathcal{A}_{\text{reg}}(F_U)$.
Then $K_{\lambda} F_U \in \check{\mathcal{B}}_{\textbf{c},\textbf{s}} \check{\mathcal{A}} \check{\mathcal{B}}_{\textbf{d},\textbf{t}}$.
\end{theorem}

\begin{proof}
Let $\sim$ be the equivalence relation modulo $\check{\mathcal{B}}_{\textbf{c},\textbf{s}} \check{\mathcal{A}} \check{\mathcal{B}}_{\textbf{d},\textbf{t}}$.
We proceed with induction on the structure of $F_U$.
If $|U| = 0$, then $K_{\lambda} F_U \in \check{\mathcal{A}} \subseteq \check{\mathcal{B}}_{\textbf{c},\textbf{s}} \check{\mathcal{A}} \check{\mathcal{B}}_{\textbf{d},\textbf{t}}$.
Let $k = |U| > 0$ and assume that, for all $V \subset U$ such that $V \neq U$, we have $K_{\lambda} F_V \in \check{\mathcal{B}}_{\textbf{c},\textbf{s}} \check{\mathcal{A}} \check{\mathcal{B}}_{\textbf{d},\textbf{t}}$.
For $u_k \in I \backslash X$ compute
\begin{equation}
\begin{split}
\label{q-radial_eqn_tedious1}
K_{\lambda} F_U 
  &= K_{\lambda} F_{u_1} F_{u_2} \ldots F_{u_{k-1}} B_{u_k}^{d_{u_k}, t_{u_k}}
    + d_{u_k} K_{\lambda} F_{u_1} F_{u_2} \ldots F_{u_{k-1}} E_{\tau(u_k)} K_{u_k}^{-1}  \\
    &\qquad - t_{u_k} K_{\lambda} F_{u_1} F_{u_2} \ldots F_{u_{k-1}} K_{u_k}^{-1}.
\end{split}
\end{equation}
By Remark \ref{q-radial_rmk_subseq} $K_{\lambda} \in \mathcal{A}_{\textrm{reg}}(F_{(u_1, u_2, \ldots, u_{k-1})})$.
Hence, by the induction hypothesis, $K_{\lambda} F_{u_1} F_{u_2} \ldots F_{u_{k-1}}$ is in $\check{\mathcal{B}}_{\textbf{c}, \textbf{s}} \check{\mathcal{A}} \check{\mathcal{B}}_{\textbf{d},\textbf{t}}$.
Therefore $K_{\lambda} F_{u_1} F_{u_2} \ldots F_{u_{k-1}} B_{u_k}^{d_{u_k}, t_{u_k}}$ belongs to $\check{\mathcal{B}}_{\textbf{c}, \textbf{s}} \check{\mathcal{A}} \check{\mathcal{B}}_{\textbf{d}, \textbf{t}}$. 
With Lemma \ref{q-radial_lem_KinBA} and the induction hypothesis we show that
\begin{equation*}
t_{u_k} K_{\lambda} F_{u_1} F_{u_2} \ldots F_{u_{k-1}} K_{u_k}^{-1} 
  = t_{u_k} q^{-(\alpha_{u_k}, \alpha_{u_1} + \alpha_{u_2} + \ldots + \alpha_{u_{k-1}})} K_{\lambda} K_{u_k}^{-1} F_{u_1} F_{u_2} \ldots F_{u_{k-1}}
\end{equation*}
is in $\check{\mathcal{B}}_{\textbf{c},\textbf{s}} \check{\mathcal{A}} \check{\mathcal{B}}_{\textbf{d}, \textbf{t}}$.
By Lemma \ref{q-radial_lem_KinBA} we write $K_{\lambda} K_{u_k}^{-1} = K_{\gamma} K_{\delta} \in \check{\mathcal{B}}_{\textbf{c}, \textbf{d}} \check{\mathcal{A}}$, where $\gamma = -\frac{1}{2}(\alpha_{u_k} - \Theta(\alpha_{u_k}))$ and $\delta = \lambda - \frac{1}{2}(\alpha_{u_k} - \Theta(\alpha_{u_k})) = \lambda - \mu_{u_k}$.
For $U' = (u_1, u_2, \ldots, u_{k-1})$ we have to show that $K_{\lambda - \mu_{u_k}} \in \mathcal{A}_{\text{reg}}(F_{U'})$ and hence, by the induction hypothesis, $K_{\lambda - \mu_{u_k}} F_{U'}$ belongs to $\check{\mathcal{B}}_{\textbf{c},\textbf{s}} \check{\mathcal{A}} \check{\mathcal{B}}_{\textbf{d}, \textbf{t}}$.
When $1 \leq \ell \leq k$ we rewrite condition \eqref{q-radial_eqn_lem_usefull} for $U$ for any subsequence $(v_1, v_2, \ldots, v_{k-\ell})$ of $U'$ to
\begin{equation} \label{q-radial_eqn_condition_Uprime}
\begin{split}
\frac{
  c_{v_1} c_{v_2} \ldots c_{v_{k - \ell}}
}{
  d_{v_1} d_{v_2} \ldots d_{v_{k - \ell}}
}
\neq
\qexp &\left( \sum_{i = 1}^{(k - 1) - (\ell - 1)} (\lambda + \mu, \alpha_{v_i} + \alpha_{\tau(v_i)})
  - \sum_{\substack{1 \leq i, j < (k - 1) - (\ell - 1), \\ i \neq j}} (\alpha_{v_i}, \alpha_{v_j}) \right).
\end{split}
\end{equation}
Substitute $\ell \mapsto \ell + 1$ in \eqref{q-radial_eqn_condition_Uprime} and take $N_k = -1$ in \eqref{q-radial_eqn_domainmu}, then 
\begin{equation*}
\begin{split}
\frac{
  c_{v_1} c_{v_2} \ldots c_{v_{k - 1 - \ell}}
}{
  d_{v_1} d_{v_2} \ldots d_{v_{k - 1 - \ell}}
}
\neq
\qexp &\left( \sum_{i = 1}^{(k - 1) - \ell} (\lambda - \mu_{u_k} + \widetilde{\mu}, \alpha_{v_i} + \alpha_{\tau(v_i)})
  - \sum_{\substack{1 \leq i, j < (k - 1) - \ell, \\ i \neq j}} (\alpha_{v_i}, \alpha_{v_j}) \right),
\end{split}
\end{equation*}
for $0 \leq \ell \leq k-1$ and where $\widetilde{\mu}$ ranges over the subset
\begin{equation*}
\{N_1 \mu_{u_1} + N_2 \mu_{u_2} + \ldots + N_{k-1} \mu_{u_{k-1}} : |N_j| \leq 1, \text{ for all } 1 \leq j \leq k-1 \}
\end{equation*}
of \eqref{q-radial_eqn_domainmu}.
Hence all the required conditions to apply the induction hypothesis on $K_{\lambda - \mu_{u_k}} F_{U'}$ hold. 
It follows that $K_{\lambda - \mu_{u_k}} F_{U'} \in \check{\mathcal{B}}_{\textbf{c},\textbf{s}} \check{\mathcal{A}} \check{\mathcal{B}}_{\textbf{d}, \textbf{t}}$ and $K_{\lambda} F_U \sim d_{u_k} K_{\lambda} F_{u_1} F_{u_2} \ldots F_{u_{k-1}} E_{\tau(u_k)} K_{u_k}^{-1}$.

We show that $E_{\tau(u_k)}K_{u_k}^{-1}$ can be pulled to the left side of (\ref{q-radial_eqn_tedious1}).
Let $p$ be the largest integer such that $\tau(u_k) = u_p$ where $0 \leq p \leq k-1$.
We show that we can pull $E_{\tau(u_k)}$ through $F_{u_p}$.
Since $[E_i, F_j] = \delta_{i,j} (q_i - q_i^{-1})^{-1} (K_i - K_i^{-1})$ we have
\begin{equation}
\label{q-radial_eqn_pulltrick}
\begin{split}
d_{u_k} &K_{\lambda} F_{u_1} F_{u_2} \ldots F_{u_{p-1}} F_{u_p} E_{\tau(u_k)} F_{u_{p+1}} \ldots F_{u_{k-1}} K_{u_k}^{-1}  \\
  &= d_{u_k} K_{\lambda} F_{u_1} F_{u_2} \ldots F_{u_{p-1}} \left(E_{\tau(u_k)} F_{u_p} - \frac{K_{u_p} - K_{u_p}^{-1}}{q_{u_p} - q_{u_p}^{-1}}\right) F_{u_{p+1}} \ldots F_{u_{k-1}} K_{u_k}^{-1}  \\
  &= d_{u_k} K_{\lambda} F_{u_1} F_{u_2} \ldots F_{u_{p-1}} E_{\tau(u_k)} F_{u_p} F_{u_{p+1}} \ldots F_{u_{k-1}} K_{u_k}^{-1}  \\
    &\qquad + 
  \qexp((\alpha_{u_p}, \alpha_{u_1} + \alpha_{u_2} + \ldots + \alpha_{u_{p-1}}) + (\alpha_{u_k}, \alpha_{u_{1}} + \alpha_{u_{2}} + \ldots + \alpha_{u_{k-1}}))  \\
    &\qquad \qquad \times (q_{u_p} - q_{u_p}^{-1})^{-1} d_{u_k} K_{\lambda} K_{u_p} K_{u_k}^{-1} F_{u_1} F_{u_2} \ldots F_{u_{p-1}} F_{u_{p+1}} \ldots F_{u_{k-1}} \\
    &\qquad - \qexp(-(\alpha_{u_p}, \alpha_{u_1} + \alpha_{u_2} + \ldots + \alpha_{u_{p-1}}) + (\alpha_{u_k}, \alpha_{u_{1}} + \alpha_{u_{2}} + \ldots + \alpha_{u_{k-1}}))  \\
    &\qquad \qquad \times (q_{u_p} - q_{u_p}^{-1})^{-1} d_{u_k} K_{\lambda} K_{u_p}^{-1} K_{u_k}^{-1} F_{u_1} F_{u_2} \ldots F_{u_{p-1}} F_{u_{p+1}} \ldots F_{u_{k-1}}.
\end{split}
\end{equation}
Let $U' = (u_1, u_2, \ldots, u_{p-1}, u_{p+1}, \ldots, u_{k-1})$.
By Lemma \ref{q-radial_lem_KinBA} we have $K_{\lambda} K_{u_p}^{\pm 1} K_{u_k}^{-1} \in \check{\mathcal{B}}_{\textbf{c},\textbf{s}} K_{\lambda \pm \mu_{u_p} - \mu_{u_k}}$.
Using the induction hypothesis we will show that $K_{\lambda \pm \mu_{u_p} - \mu_{u_k}} F_{U'} \in \check{\mathcal{B}}_{\textbf{c},\textbf{s}} \check{\mathcal{A}} \check{\mathcal{B}}_{\textbf{c}, \textbf{s}}$, so that \eqref{q-radial_eqn_pulltrick} gives
\begin{equation} \label{q-radial_eqn_pulltrick_equiv}
\begin{split}
d_{u_k} &K_{\lambda} F_{u_1} F_{u_2} \ldots F_{u_{p-1}} F_{u_p} E_{\tau(u_k)} F_{u_{p+1}} \ldots F_{u_{k-1}} K_{u_k}^{-1}  \\
 &\sim d_{u_k} K_{\lambda} F_{u_1} F_{u_2} \ldots F_{u_{p-1}} E_{\tau(u_k)} F_{u_p} F_{u_{p+1}} \ldots F_{u_{k-1}} K_{u_k}^{-1}.
\end{split}
\end{equation}
We show that $K_{\lambda \pm \mu_{u_p} - \mu_{u_k}} \in \mathcal{A}_{\text{reg}}(F_{U'})$
Note that, for $2 \leq \ell \leq k$, and for any subsequence $(v_1, v_2, \ldots, v_{k-\ell})$ of $U'$, \eqref{q-radial_eqn_domainmu} gives
\begin{equation} \label{q-radial_eqn_condition_Uprime2}
\begin{split}
\frac{
  c_{v_1} c_{v_2} \ldots c_{v_{k - \ell}}
}{
  d_{v_1} d_{v_2} \ldots d_{v_{k - \ell}}
}
\neq
\qexp &\left( \sum_{i = 1}^{(k - 2) - (\ell - 2)}(\lambda + \mu,  \alpha_{v_i} + \alpha_{\tau(v_i)})
  - \sum_{\substack{1 \leq i, j < (k - 2) - (\ell - 2), \\ i \neq j}} (\alpha_{v_i}, \alpha_{v_j}) \right).
\end{split}
\end{equation}
Substitute $\ell \mapsto \ell + 2$ in \eqref{q-radial_eqn_condition_Uprime2} and take $N_k = -1$, $N_{p} = \pm 1$ in \eqref{q-radial_eqn_domainmu}, then
\begin{equation*}
\begin{split}
\frac{
  c_{v_1} c_{v_2} \ldots c_{v_{k - 2 - \ell}}
}{
  d_{v_1} d_{v_2} \ldots d_{v_{k - 2 - \ell}}
}
\neq
\qexp &\left( \sum_{i = 1}^{(k - 2) - \ell} (\lambda \pm \mu_{u_p} - \mu_{u_k} + \widetilde{\mu}, \alpha_{v_i} + \alpha_{\tau(v_i)})
  - \sum_{\substack{1 \leq i, j < (k - 2) - \ell, \\ i \neq j}} (\alpha_{v_i}, \alpha_{v_j}) \right),
\end{split}
\end{equation*}
for $0 \leq \ell \leq k-2$ and where $\widetilde{\mu}$ ranges over the subset
\begin{equation*}
\{N_1 \mu_{u_1} + \ldots + N_{p-1} \mu_{u_{p-1}} + N_{p+1} \mu_{u_{p+1}} + \ldots + N_{k-1} \mu_{u_{k-1}} : |N_j| \leq 1, \text{ for all } 1 \leq j \leq k-1, j \neq p \}
\end{equation*}
of \eqref{q-radial_eqn_domainmu}.
Hence all the required conditions to apply the induction hypothesis on $K_{\lambda \pm \mu_{u_p} - \mu_{u_k}} F_{U'}$ hold. 
It follows that $K_{\lambda - \mu_{u_k}} F_{U'} \in \check{\mathcal{B}}_{\textbf{c},\textbf{s}} \check{\mathcal{A}} \check{\mathcal{B}}_{\textbf{d}, \textbf{t}}$.

By repeated application of (\ref{q-radial_eqn_pulltrick_equiv}) we can pull $E_{\tau(u_k)}$ through (\ref{q-radial_eqn_tedious1}).
Taking into account the $q$-commutation relations for $K_{u_k}^{-1}$ and $F_{u_i}$ we have modulo $\check{\mathcal{B}}_{\textbf{c}, \textbf{s}} \check{\mathcal{A}} \check{\mathcal{B}}_{\textbf{d}, \textbf{t}}$
\begin{equation}
\label{q-radial_eqn_pullEvi}
\begin{split}
K_{\lambda} F_{U} 
  &\sim \qexp((\lambda, \alpha_{\tau(u_k)}) - (\alpha_{u_k}, \alpha_{u_1} + \alpha_{u_2} + \ldots + \alpha_{u_{k-1}}))  \\
    &\qquad \qquad \times \frac{d_{u_k}}{c_{u_k}} c_{u_k} E_{\tau(u_k)} K_{u_k}^{-1} 
  K_{\lambda} F_{u_1} F_{u_2} \ldots F_{u_{k-1}}  \\
  &\quad = \qexp((\lambda, \alpha_{\tau(u_k)}) - (\alpha_{u_k}, \alpha_{u_1} + \alpha_{u_2} + \ldots + \alpha_{u_{k-1}})) \\
    &\qquad \qquad \times \frac{d_{u_k}}{c_{u_k}} (F_{u_k} - c_{u_k} B_{u_k}^{c_{u_k}, s_{u_k}} + s_{u_k} K_{u_k}^{-1})
    K_{\lambda} F_{u_1} F_{u_2} \ldots F_{u_{k-1}}  \\
  &\quad \sim \qexp((\lambda, \alpha_{u_k} + \alpha_{\tau(u_k)}) - (\alpha_{u_k}, \alpha_{u_1} + \alpha_{u_2} + \ldots + \alpha_{u_{k-1}}))  \\
    &\qquad \qquad \times \frac{d_{u_k}}{c_{u_k}} K_{\lambda} F_{u_k} F_{u_1} F_{u_2} \ldots F_{u_{k-1}},  
\end{split}
\end{equation}
because we already noticed that $K_{u_k}^{-1} K_{\lambda} F_{u_1} F_{u_2} \ldots F_{u_{k-1}} \in \check{\mathcal{B}}_{\textbf{c}, \textbf{s}} \check{\mathcal{A}} \check{\mathcal{B}}_{\textbf{d}, \textbf{t}}$.
Hence we obtain for $U \in (I \backslash X)^k$ and $K_{\lambda} \in \mathcal{A}_{\text{reg}}(F_U)$ the identity
\begin{equation}
\label{q-radial_eqn_permutation}
\begin{split}
K_{\lambda} F_{U} &\sim C(\lambda, U) K_{\lambda} F_{u_k} F_{u_1} F_{u_2} \ldots F_{u_{k-1}},
\end{split}
\end{equation}
where
\begin{equation}
\label{q-radial_eqn_permutation_coef}
\begin{split}
C(\lambda, U) &= \qexp((\lambda, \alpha_{u_k} + \alpha_{\tau(u_k)}) - (\alpha_{u_k}, \alpha_{u_1} + \alpha_{u_2} + \ldots + \alpha_{u_{k-1}})) \frac{d_{u_k}}{c_{u_k}} 
\end{split}
\end{equation}

Let $S_k$ be the permutation group on $\{1, 2, \ldots, k\}$.
For $V = (v_1, v_2, \ldots, v_k) \in I^k$ and $\sigma \in S_k$ define the action of $\sigma$ on $I^k$ by $\sigma V = (v_{\sigma(1)}, v_{\sigma(2)}, \ldots, v_{\sigma(k)})$.

Let $\sigma = (1\,2\,\ldots\,k) \in S_k$ be the rotation of order $k$.
Define $U_0 = U$ and $U_{\ell} = \sigma U_{\ell-1}$ for $\ell > 0$.
Fix $K_{\lambda} \in \mathcal{A}_{\text{reg}}(F_U)$.
Note that requirements \eqref{q-radial_eqn_lem_usefull} and \eqref{q-radial_eqn_domainmu} are invariant under the action of $\sigma$, i.e. $v_i \mapsto v_{\sigma(i)}$ for all $i \in I$ in \eqref{q-radial_eqn_lem_usefull} and \eqref{q-radial_eqn_domainmu}, hence for all $\ell \geq 0$ we have $K_{\lambda} \in \mathcal{A}_{\text{reg}}(F_{U_{\ell}})$.
Therefore the requirements for \eqref{q-radial_eqn_permutation} are satisfied for all $U_{\ell}$ and we have $K_{\lambda} F_{U_{\ell}} \sim C(\lambda, U_{\ell}) F_{U_{\ell+1}}$ for all $\ell \geq 0$.
Since $\sigma^k = \id$ it follows that $U_k = U_0$ and from \eqref{q-radial_eqn_permutation_coef} we have
\begin{equation} \label{q-radial_eqn_final_product}
\begin{split}
K_{\lambda} F_{U} &\sim \left( \prod_{\ell = 0}^{k-1} C(\lambda, U_i) \right) K_{\lambda} F_U \\
  &= \qexp \left( (\lambda, \sum_{i = 1}^k (\alpha_{u_i} + \alpha_{u_{\tau(i)}}))
    - \sum_{\substack{1 \leq i, j \leq k, \\ i \neq j}} (\alpha_{u_i}, \alpha_{u_j}) \right) \times \frac{d_{u_1} d_{u_2} \ldots d_{u_k}}{c_{u_1} c_{u_2} \ldots c_{u_k}} K_{\lambda} F_U. 
\end{split}
\end{equation}
Because $K_{\lambda} \in \mathcal{A}_{\text{reg}}(F_U)$ it follows from \eqref{q-radial_eqn_lem_usefull} that the coefficient of $F_U$ on the right hand side of \eqref{q-radial_eqn_final_product} is not equal to one.
Subtracting the right hand side gives $K_{\lambda} F_U \sim 0$ or equivalent $K_{\lambda} F_U \in \check{\mathcal{B}}_{\textbf{c}, \textbf{s}} \check{\mathcal{A}} \check{\mathcal{B}}_{\textbf{d}, \textbf{t}}$.
\end{proof}

\begin{remark}
The proof of Theorem \ref{q-radial_thm_main} is constructive, hence provides an algorithm to calculate the radial part $K_{\lambda} F_U$ for all $U = (u_1, u_2, \ldots, u_k) \in (I \backslash X)^k$ where $K_{\lambda} \in \mathcal{A}_{\textrm{reg}}(F_U)$.
However the number of terms in $\check{\mathcal{B}}_{\textbf{c}, \textbf{s}} \check{\mathcal{A}} \check{\mathcal{B}}_{\textbf{d}, \textbf{t}}$ of $K_{\lambda} F_U$ where $|U| = k$ grows exponentially in $k$.
Indeed, in the worst case, to permute $U = (u_1, u_2, \ldots, u_k)$ to $(u_k, u_1, u_2, \ldots, u_{k-1})$, \eqref{q-radial_eqn_tedious1} gives two extra terms $K_{\lambda}F_V$, where $|V| = k-1$, \eqref{q-radial_eqn_pulltrick} gives $2(k-2)$ extra terms $K_{\lambda} F_{V'}$, where $|V'| = k-2$, and \eqref{q-radial_eqn_pullEvi} gives two extra terms $K_{\lambda} F_V$, where $|V| = k-1$.
If $f_k$ is the number of terms for $K_{\lambda}F_U$, where $|U| = k$, in the worst case, we have recurrence relation $f_k = 4f_{k-1} + 2(k-2)f_{k-2}$ for $k \geq 1$, with starting values $f_{-1} = 0$ and $f_0 = 1$.
Since $f_k \geq 4f_{k-1}$ we have $f_k \geq 4^k$.
On the other hand, $f_{k-1} \geq f_{k-2}$, so that $f_k = 4f_{k-1} + 2(k-2)f_{k-2} \leq 2kf_{k-1}$.
Therefore an upper bound is given by $f_k \leq 2^k k!$.
This shows that $f_k$ grows exponentially in $k$.
Applying the permutation $k$ times finishes the algorithm, therefore, in the worst case, producing $kf_k$ terms.
The number of terms $kf_k$ grows exponentially in $k$.
\end{remark}

\begin{remark}
Assumption (\ref{q-radial_eqn_lem_usefull}) in Theorem \ref{q-radial_thm_main} is to be expected.
A similar assumption is made in \cite[Theorem 2.4]{CM82} where Casselman and Mili\v{c}i\'c take the action of $a$, where $a$ is an regular point of the torus $A$.
If $a$ would not be a regular point \cite[Lemma 2.2]{CM82} fails to be true because of a division by zero.
The same problem occurs in the quantum case if we don't assume (\ref{q-radial_eqn_lem_usefull}).
\end{remark}

\section{Spherical functions on quantum symmetric pairs} \label{q-radial_sec_sphericalfunctions}

In this section we introduce spherical functions in general.
We give a quantum analogue for the map $\Pi$ of Casselman and Mili\v{c}i{\'c} \cite{CM82} for quantum symmetric pairs with simple generators.
Moreover, we prove a quantum analogue, Theorem \ref{q-radial_thm_CM_analogue}, of \cite[Theorem 3.1]{CM82}.
Next we study the $*$-invariance of the right coideals of a Hopf $*$-algebra.
Theorem \ref{q-radial_thm_bi_invariant_weight} gives conditions for an orthogonality relation for spherical functions on Hopf $*$-algebras.

\begin{definition}
Let $\mathcal{B}$ and $\mathcal{B}'$ be two right coideal subalgebras of a Hopf algebra $A$.
Take finite dimensional representations $t_{\mathcal{B}}$ and $t_{\mathcal{B}'}$ of $\mathcal{B}$, $\mathcal{B}'$ acting on vector spaces $V$, $W$ respectively.
A linear map $\Phi : A \to \End(W, V)$ is called a $(t_{\mathcal{B}}, t_{\mathcal{B}'})$-spherical function if for all $B \in \mathcal{B}$, $B' \in \mathcal{B}'$ and $Y \in A$
\begin{equation*}
\Phi(B Y B') = t_{\mathcal{B}}(B) \Phi(Y) t_{\mathcal{B}'}(B').
\end{equation*}
The space of all $(t_{\mathcal{B}}, t_{\mathcal{B}'})$-spherical functions is denoted by $\mathcal{F}_{t_{\mathcal{B}}, t_{\mathcal{B}'}}(A)$.
\end{definition}

\begin{remark}
The case $A = \mathcal{U}_q(\mathfrak{g})$, $\mathcal{B} = \mathcal{B}_{\textbf{c}, \textbf{s}}$ and $\mathcal{B}' = \mathcal{B}_{\textbf{d}, \textbf{t}}$ where $\mathfrak{g}$ is a semi-simple Lie algebra with non-reduced root system has been studied by Letzter \cite{Let04}.
She classifies all spherical functions where $t_{\mathcal{B}} = \epsilon_{\mathcal{B}}$ and $t_{\mathcal{B}'} = \epsilon_{\mathcal{B}'}$.
In this case the $(\epsilon_{\mathcal{B}}, \epsilon_{\mathcal{B}'})$-spherical functions are identified with Macdonald polynomials, \cite[Theorem 8.2]{Let04}.
\end{remark}

\begin{definition} \label{q-radial_def_radial_part}
Let $(X,\tau)$ be an admissible pair and $\check{\mathcal{B}}_{\textbf{c},\textbf{s}} = \check{\mathcal{B}}_{\textbf{c},\textbf{s}}(X,\tau)$ and $\check{\mathcal{B}}_{\textbf{d},\textbf{t}} = \check{\mathcal{B}}_{\textbf{c}, \textbf{t}}(X,\tau)$ be two right coideals of $\check{\mathcal{U}}_q(\mathfrak{g}')$ with simple generators.
Fix finite dimensional representations $t_1$ of $\check{\mathcal{B}}_{\textbf{c},\textbf{s}}$ acting on vector space $V$ and $t_2$ of $\check{\mathcal{B}}_{\textbf{d},\textbf{t}}$ acting on vector space $W$.
Let $\mathcal{F}_{t_1, t_2}(\check{\mathcal{U}}_q(\mathfrak{g}'))$ be the set of $(t_1, t_2)$-spherical functions on $\check{\mathcal{U}}_q(\mathfrak{g}')$.
Write $\textrm{Res}$ for the restriction map of $\mathcal{F}_{t_1, t_2}(\check{\mathcal{U}}_q(\mathfrak{g}'))$ to the quantum torus $\check{\mathcal{A}}$, i.e. $\textrm{Res}(\Phi) : \check{\mathcal{A}} \to \End(W, V)$.
We define an action of $\Uq(\mathfrak{g})$ on $\mathcal{F}_{t_1, t_2}$ by $Y.\Phi(Z) = \Phi(ZY)$ for all $Y, Z \in \Uq(\mathfrak{g})$ and $\Phi \in \mathcal{F}_{t_1, t_2}(\Uq(\mathfrak{g}'))$.

Let $Y \in \check{U}_q(\mathfrak{g}')$ and $K_{\lambda} \in \mathcal{A}_{\text{reg}}(Y)$.
According to Theorem \ref{q-radial_thm_quantumiwasawa} we have $K_{\lambda} Y \in \check{\mathcal{B}}_{\textbf{c}, \textbf{s}} \check{\mathcal{A}} \mathcal{N}$, hence we can write $K_{\lambda}Y = \sum_{i} C_i(\lambda; Y) K_{\mu_i(\lambda;Y)} F_{V_i(\lambda;Y)}$, where $C_i(\lambda;Y) \in \check{\mathcal{B}}_{\textbf{c}, \textbf{s}}$, $\mu_i(\lambda;Y) \in \check{Q}$ with $\Theta(\mu_i(\lambda;Y)) = -\mu_i(\lambda;Y)$ and $V_i(\lambda;Y) \in \bigcup_{n \geq 0} (I \backslash X)^n$. 
Combining Theorem \ref{q-radial_thm_quantumiwasawa} with Theorem \ref{q-radial_thm_main} applied on all $K_{\mu_i(\lambda;Y)} F_{V_i(\lambda;V)}$ for $K_{\lambda} \in \mathcal{A}_{\textrm{reg}}(Y)$ we have the map
\begin{equation*}
\Pi(Y) : \mathcal{A}_{\textrm{reg}}(Y) \to \check{\mathcal{B}}_{\textbf{c}, \textbf{s}} \tensor \check{\mathcal{A}} \tensor \check{\mathcal{B}}_{\textbf{d}, \textbf{t}}.
\end{equation*}
For $K_{\lambda} \in \mathcal{A}_{\text{reg}}(Y)$, we introduce the notation
\begin{equation} \label{q-radial_eqn_explicit_BAB}
\Pi(Y)(K_{\lambda}) = \sum_{i} B_i(\lambda;Y) K_{\lambda_i(\lambda;X)} B_i'(\lambda;Y),
\end{equation}
where $B_i(\lambda;Y) \in \check{\mathcal{B}}_{\textbf{c}, \textbf{s}}$, $\lambda_i(\lambda;X) \in \check{\mathcal{A}}$ and $B_i'(\lambda;Y) \in \check{\mathcal{B}}_{\textbf{d}, \textbf{t}}$.
Let $\eta_{t_1, t_2} = t_1 \tensor 1 \tensor t_2$ and define the map
\begin{equation*}
\Pi_{t_1, t_2}(Y) = \eta_{t_1, t_2} \circ \Pi(Y) : \mathcal{A}_{\textrm{reg}}(Y) \to \End(V) \tensor \check{\mathcal{A}} \tensor \End(W).
\end{equation*}
We define a map $\cdot : \End(V) \tensor \check{\mathcal{A}} \tensor \End(W) \times \textrm{Res}(\mathcal{F}_{t_1, t_2}(\check{\mathcal{U}}_q(\mathfrak{g}')) \to \End(W, V)$ defined on the generators by
\begin{equation*}
(T_1 \tensor K_{\lambda} \tensor T_2) \cdot \textrm{Res}(\Phi) = T_1 \Phi(K_{\lambda}) T_2,
\end{equation*}
for every $T_1 \in \End(V)$, $T_2 \in \End(W)$, $\lambda \in \check{Q}$ such that $\Theta(\lambda) = -\lambda$ and $\Phi \in \mathcal{F}_{t_1, t_2}(\check{\mathcal{U}}_q(\mathfrak{g}')$.
\end{definition}

The following Theorem is a quantum analogue of \cite[Theorem 3.1]{CM82}.

\begin{theorem} \label{q-radial_thm_CM_analogue}
Let $\check{\mathcal{B}}_{\textbf{c},\textbf{s}}$ and $\check{\mathcal{B}}_{\textbf{d},\textbf{t}}$ be two right coideal with simple generators of $\check{\mathcal{U}}_q(\mathfrak{g}')$.
Let $Y \in \mathcal{U}_q(\mathfrak{g}')$ and $\Phi \in \mathcal{F}_{t_1, t_2}(\check{\mathcal{U}}_q(\mathfrak{g}'))$.
We have $\textrm{Res}(Y.\Phi)(K_{\lambda}) = \Pi_{t_1, t_2}(Y)(K_{\lambda}) \cdot \textrm{Res}(\Phi)$, for all $K_{\lambda} \in \mathcal{A}_{\textrm{reg}}(Y)$.
\end{theorem}

\begin{proof}
Let $K_{\lambda} \in \mathcal{A}_{\textrm{reg}}(Y)$ and by \eqref{q-radial_eqn_explicit_BAB} write $\Pi(Y)(K_{\lambda}) = \sum_{i} B_i(\lambda;Y) K_{\lambda_i(\lambda;X)} B_i'(\lambda;Y)$.
Evaluate both sides in $K_{\lambda}$, which gives
\begin{equation*}
\textrm{Res}(Y.\Phi)(K_{\lambda}) 
  = \Phi(K_{\lambda}Y)
  = \sum_{i} t_1(B_i(\lambda;Y)) \Phi(K_{\mu_i(\lambda;Y)}) t_2(B_i'(\lambda;Y)),
\end{equation*}
and on the other hand
\begin{equation*}
\begin{split}
\Pi_{t_1, t_2}(Y)(K_{\lambda}) \cdot \textrm{Res}(\Phi)
  &= (\sum_i t_1(B_i(\lambda;Y)) \tensor K_{\mu_i(\lambda;Y)} \tensor t_2(B_i'(\lambda;Y))) \cdot \Phi \\
  &= \sum_i t_1(B_i(\lambda;Y)) \Phi(K_{\mu_i(\lambda;Y)}) t_2(B_i'(\lambda;Y)).
\qedhere
\end{split}
\end{equation*}
\end{proof}

The element $\Pi_{t_1, t_2}(Y)$ is called the radial part of $Y$.
In Section \ref{q-radial_sec_applications} we compute the radial part of central elements in three examples.
But first we study the $*$-invariance of $\mathcal{B}_{\textbf{c}, \textbf{s}}$.

\begin{proposition} \label{q-radial_prop_star}
Let $(X, \tau)$ be an admissible pair and $\mathcal{B}_{\textbf{c}, \textbf{s}} = \mathcal{B}_{\textbf{c}, \textbf{s}}(X, \tau)$ be a right coideal subalgebra of $\Uq(\mathfrak{g}')$ with simple generators.
Let $\mu \in \Aut(A,X)$ such that $\mu$ restricted to $X$ is the identity on $X$, $\mu^2 = \id$ and $\mu \tau = \tau \mu$.
Define a complex $*$-structure on $\Uq(\mathfrak{g}')$ by
\begin{equation} \label{q-radial_eqn_*-structure}
K_i^* = K_{\mu(i)}, \quad E_i^* = \sigma_i K_{\mu(i)} F_{\mu(i)}, \quad F_i^* = \sigma_i E_{\mu(i)} K_{\mu(i)}^{-1},
\end{equation}
where $i \in I$, $\sigma_i = 1$ if $\mu(i) \neq i$ and $\sigma_i \in \{ \pm 1 \}$ if $\mu(i) = i$.
Let $\mathcal{B}_{\textbf{c},\textbf{s}}$ be a right coideal with simple generators such that
\begin{equation} \label{q-radial_eqn_*-invBrequirements}
c_{\mu\tau(i)} = \sigma_i \sigma_{\tau(i)} \overline{c_i}^{-1} q^{2-(\alpha_{\mu(i)}, \alpha_{\mu\tau(i)})},
\quad s_{\mu\tau(i)} = -\sigma_i \overline{c_i}^{-1} \overline{s_i}.
\end{equation}
The $*$-structure in (\ref{q-radial_eqn_*-structure}) is a $*$-operator on $\Uq(\mathfrak{g}')$.
Moreover $\mathcal{B}_{\textbf{c}, \textbf{s}}$ is $*$-invariant with the $*$-action given on the generators $\alpha \in Q$, with $\Theta(\alpha) = \alpha$, and $i \in I \backslash X$ by
\begin{equation} \label{q-radial_eqn_*-explicit-B}
K_{\alpha}^* = K_{\beta}, \quad B_i^* = -\sigma_i \overline{c_i} K_{\mu\tau(i) - \mu(i)} B_{\mu\tau(i)},
\end{equation}
where $\alpha = \sum_{i \in I} n_i \alpha_i \in Q$, $\Theta(\alpha) = \alpha$ and $\beta = \sum_{i \in I} n_i \alpha_{\mu(i)}$.
\end{proposition}

\begin{proof}
By a direct verification on the generators (\ref{q-radial_eqn_*-structure}) defines a $*$-structure on $\Uq(\mathfrak{g}')$, see also \cite[Proposition 6.1.17]{KS97}.
Assuming (\ref{q-radial_eqn_*-invBrequirements}) a straightforward calculation gives (\ref{q-radial_eqn_*-explicit-B}).
We check that the elements of \eqref{q-radial_eqn_*-explicit-B} are again in $\mathcal{B}_{\textbf{c}, \textbf{s}}$.
Because $\mathcal{B}_{\textbf{c}, \textbf{s}}$ is simple, by Corollary \ref{q-radial_cor_alternative_simple_B} and the definition of $\Theta$, see \cite[(2.10)]{Kol14}, we have $\Theta(\alpha_i) = -\alpha_{\tau(i)}$ if $i \in I \backslash X$ and $\Theta(\alpha_j) = \alpha_j$ if $j \in X$.
Note that $\mu(j) = j$ for all $j \in X$, $\tau^2 = \id$ and $\tau \mu = \mu \tau$.
Let $\alpha \in Q$ such that $\Theta(\alpha) = \alpha$ and write $\alpha = \sum_{i \in I} n_i \alpha_i$, then
\begin{equation*}
\sum_{i \in I} n_i \alpha_i = \Theta\left( \sum_{i \in I} n_i \alpha_i \right) = - \sum_{i \in I \backslash X} n_{\tau(i)} \alpha_i + \sum_{j \in X} n_j \alpha_j,
\end{equation*}
so that $n_i = -n_{\tau(i)}$ if $i \in I \backslash X$.
Let $\beta = \sum_{i \in I} n_i \alpha_{\mu(i)}$, such that $K_{\alpha}^* = K_{\beta}$, then
\begin{equation*}
\Theta(\beta) = - \sum_{i \in I \backslash X} n_i \alpha_{\tau \mu(i)} + \sum_{j \in X} n_j \alpha_{\mu(j)}
  = - \sum_{i \in I \backslash X} n_i \alpha_{\mu \tau(i)} + \sum_{j \in X} n_j \alpha_j
  = - \sum_{i \in I \backslash X} n_{\tau(i)} \alpha_{\mu(i)} + \sum_{j \in X} n_j \alpha_j
  = \beta,
\end{equation*}
so that $K_{\beta} \in \mathcal{B}_{\textbf{c}, \textbf{s}}$.
Let $i \in I \backslash X$, then by \eqref{q-radial_eqn_*-explicit-B} $B_i^* = -\sigma_i \overline{c_i} K_{\mu\tau(i) - \mu(i)} B_{\mu\tau(i)}$.
We show that $K_{\mu\tau(i) - \mu(i)} \in \mathcal{B}_{\textbf{c}, \textbf{s}}$.
We have $\Theta(\alpha_{\mu\tau(i)} - \alpha_{\mu(i)}) = \alpha_{\tau\mu(i)} - \alpha_{\tau\mu\tau(i)} = \alpha_{\mu\tau(i)} - \alpha_{\mu(i)}$, hence $K_{\mu\tau(i) - \mu(i)} \in \mathcal{B}_{\textbf{c}, \textbf{s}}$.
Therefore $B_i^* \in \mathcal{B}_{\textbf{c}, \textbf{s}}$, when $i \in I \backslash X$.
For $i \in X$ we have $E_i^*, F_i^*, K_i^* \in \mathcal{M}_X \subseteq \mathcal{B}_{\textbf{s}, \textbf{t}}$.
This yields that $\mathcal{B}_{\textbf{c}, \textbf{s}}^* = \mathcal{B}_{\textbf{c}, \textbf{s}}$, which proves the last part of the proposition.
\end{proof}

\begin{remark} \label{q-radial_rmk_star}
Take $\mu = \id$ and $\sigma_i = 1$ for all $i \in I$ in Proposition \ref{q-radial_prop_star}, then the $*$-operator defined by $K_i^* = K_i$, $E_i^* = K_i F_i$ and $F_i^* = E_i K_i^{-1}$ is called the compact real form of $\Uq(\mathfrak{g}')$.
For the right coideal $\mathcal{B}_{\textbf{c}, \textbf{s}}$ such that $c_{\tau(i)} = \overline{c_i}^{-1} q^{2-(\alpha_{i}, \alpha_{\tau(i)})}$ and $s_{\tau(i)} = -\overline{c_i}^{-1} \overline{s_i}$, we have that $\mathcal{B}_{\textbf{c}, \textbf{s}}$ is $*$-invariant.
We denote $\Uq(\mathfrak{su}_n)$ for the quantized universal enveloping algebra $\Uq(\mathfrak{sl}_n)$ equipped with the compact real form.
\end{remark}

\begin{definition}
For a Hopf ($*$-)algebra $A$ define the set $A_{\textrm{grp}}$ of invertible group-like elements to be the invertible elements $a \in A$ such that $\Delta(a) = a \tensor a$.
Note that from the Hopf algebra axioms we have $m \circ (\epsilon \tensor \id) \circ \Delta = \id$, thus for every $a \in A_{\textrm{grp}}$ we have $a = \epsilon(a) a$.
Since every $a \in A_{\textrm{grp}}$ is invertible it follows that $\epsilon(a) = 1$.
\end{definition}

Let $A$ be a Hopf algebra and let $\mathcal{B}$, $\mathcal{B}'$ be two right coideal subalgebras of $A$.
Let $t_{\mathcal{B}}$ be a finite dimensional representations of $\mathcal{B}$ and let $t_{\mathcal{B}'}$ be a finite dimensional representations of $\mathcal{B}'$.
We define a right action of $A$ on $\mathcal{F}_{t_{\mathcal{B}}, t_{\mathcal{B}'}}(A)$ by $(\Phi.a)(a') = \Phi(aa')$ for all $a, a' \in A$ and $\Phi \in \mathcal{F}_{t_{\mathcal{B}}, t_{\mathcal{B}'}}(A)$.

Theorem \ref{q-radial_thm_bi_invariant_weight} motivates the $*$-invariance of the right coideal subalgebras on $\Uq(\mathfrak{g})$.
This theorem is a generalization of \cite[Theorem 5.5]{AKR15}, which plays a key role in the orthogonality for the matrix valued spherical functions of the quantum symmetric pair $(\Uq(\mathfrak{su}_2) \tensor \Uq(\mathfrak{su}_2), \mathcal{B})$.

\begin{theorem} \label{q-radial_thm_bi_invariant_weight}
Let $A$ be a Hopf $*$-algebra and $\mathcal{B}$, $\mathcal{B}'$ be two $*$-invariant right coideal subalgebras of $A$.
Let $\Phi, \Psi : A \to \End(V, W)$ be two $(t^{\mathcal{B}}, t^{\mathcal{B}'})$-spherical functions for unitary representations $t^{\mathcal{B}}$ of $\mathcal{B}$ and $t^{\mathcal{B}'}$ of $\mathcal{B}'$.
Take $a_s \in A_{\textrm{grp}}$ to be a self adjoint element, i.e. $a_s^* = a_s$, and define
\begin{equation*}
\tau_{a_s} : A \to \CC : a \mapsto \tr((\Phi . a_s)(\Psi . a_s)^*).
\end{equation*}
If $a_s^2 S(a^*) = S(a)^* a_s^2$ for all $a \in A$ then $\tau_{a_s}$ is a $(\epsilon_{a_s^{-1}\mathcal{B}a_s}, \epsilon_{\mathcal{B}'})$-spherical function, where $\epsilon_{a_s^{-1}\mathcal{B}a_s}$ is the counit representation restricted to $a_s^{-1} \mathcal{B} a_s$ and $\epsilon_{\mathcal{B}'}$ the counit representation of $\mathcal{B}'$.
We reformulate that if $a_s^2 S(a^*) = S(a)^* a_s^2$ for all $a \in A$, then for all $a \in A$, $b \in \mathcal{B}$ and $b' \in \mathcal{B}'$ we have
\begin{equation*}
\tau_{a_s}(a_s^{-1} b a_s a b') = \epsilon(a_s^{-1} b a_s) \tau_{a_s}(a) \epsilon(b').
\end{equation*}
\end{theorem}

\begin{remark}
The algebra $a_s^{-1} \mathcal{B} a_s$ in Theorem \ref{q-radial_thm_bi_invariant_weight} is a right coideal subalgebra, because $a_s$ and $a_s^{-1}$ are group like.
Also if $a_s$ does not satisfy $a_s^2 S(a^*) = S(a)^* a_s^2$ for all $a \in A$ the spherical function behavior $\tau_{a_s}(ab') = \tau_{a_s}(a) \epsilon(b')$ still holds for all $a \in A$ and $b \in \mathcal{B}'$.
\end{remark}

\begin{proof}[Proof of Theorem \ref{q-radial_thm_bi_invariant_weight}]
Let $M = \dim(V)$ and $N = \dim(W)$.
Write $\Phi = (\Phi_{m,n})_{m,n}$ and $\Psi = (\Psi_{m,n})_{m,n}$ with $\Phi_{m,n}, \Psi_{m,n}$ linear functionals on $A$ where $1 \leq m \leq M$ and $1 \leq n \leq N$.
A straightforward calculation gives
\begin{equation*}
\tr((\Phi.a_s)(\Psi.a_s)^*) = \sum_{m = 1}^M \sum_{n = 1}^N (\Phi_{m,n} . a_s) (\Psi_{m,n} . a_s)^*,
\end{equation*}
so that for every $a \in A$
\begin{equation*}
\tau_{a_s}(a) = \sum_{m = 1}^M \sum_{n = 1}^N \sum_{(a)} \Phi_{m,n}(a_s a_{(1)}) \overline{\Psi_{m,n}(a_s S(a_{(2)})^*)},
\end{equation*}
using $\xi^*(a) = \xi(S(a)^*)$ for all linear functions $\xi : A \to \CC$ and $a \in A$. 

Take $b' \in \mathcal{B}'$ and $a \in A$ we prove first that $\tau_{a_s}(ab') = \tau_{a_s}(a) \epsilon(b')$ for all $a_s \in A_{\textrm{grp}}$.
We have
\begin{equation*}
\tau_{a_s}(ab') = \sum_{m = 1}^M \sum_{n = 1}^N \sum_{(a), (b')} \Phi_{m,n}(a_s a_{(1)} b'_{(1)}) \overline{\Psi_{m,n}(a_s S(a_{(2)})^* S(b'_{(2)})^*)}.
\end{equation*}
Because $\mathcal{B}'$ is a right coideal of $A$ we have $b'_{(1)} \in \mathcal{B}'$.
Since $\Phi$ is a $(t^{\mathcal{B}}, t^{\mathcal{B}'})$-spherical function we have $\Phi_{m,n}(a_s a_{(1)} b'_{(1)}) = \sum_{k = 1}^N \Phi_{m,k}(a_s a_{(1)}) t^{\mathcal{B}'}_{k,n}(b'_{(1)})$.
Moreover $t^{\mathcal{B}'}$ is a unitary representation and $\mathcal{B}'$ is $*$-invariant, hence $t^{\mathcal{B}'}_{k,n}(b'_{(1)}) = \overline{t^{\mathcal{B}'}_{n,k}((b'_{(1)})^*)}$.
We obtain
\begin{equation*}
\begin{split}
\tau_{a_s}(ab') &= \sum_{m = 1}^M \sum_{n, k = 1}^N \sum_{(a), (b')} \Phi_{m,k}(a_s a_{(1)}) t^{\mathcal{B}'}_{k,n}(c_{(1)}) \overline{\Psi_{m,n}(a_s S(a_{(2)})^* S(b'_{(2)})^*)} \\
  &= \sum_{m = 1}^M \sum_{n, k = 1}^N \sum_{(a), (b')} \Phi_{m,k}(a_s a_{(1)}) \overline{\Psi_{m,n}(a_s S(a_{(2)})^* S(b'_{(2)})^*) t^{\mathcal{B}'}_{n,k}((b'_{(1)})^*)} \\
  &= \sum_{m = 1}^M \sum_{k = 1}^N \sum_{(a), (b')} \Phi_{m,k}(a_s a_{(1)}) \overline{\Psi_{m, k}(a_s S(a_{(2)})^* S(b'_{(2)})^* (b'_{(1)})^*)} \\
  &= \sum_{m = 1}^M \sum_{k = 1}^N \sum_{(a)} \Phi_{m,k}(a_s a_{(1)}) \overline{\Psi_{m, k}\left(a_s S(a_{(2)})^* \sum_{(b')} S(b'_{(2)})^* (b'_{(1)})^* \right)}.
\end{split}
\end{equation*}
Use the antipode axiom for the Hopf algebra $A$ so that $\sum_{(b')} S(b'_{(2)})^* (b'_{(1)})^* = (\sum_{(b')} b'_{(1)} S(b'_{(2)}))^* = \overline{\epsilon(b')}$.
Therefore we see that $\tau_{a_s}(ab') = \tau_{a_s}(a) \epsilon(b')$.

Now we show that if $a_s^2 S(a^*) = S(a)^* a_s^2$ for all $a \in A$ then $\tau_{a_s}(a_s^{-1} b a_s a) = \epsilon(a_s^{-1} b a_s)\tau_{a_s}(a)$ for all $a \in A$ and $b \in \mathcal{B}$.
Let $a \in A$ and $b \in \mathcal{B}$, we have
\begin{equation*}
\tau(a_s^{-1} b a_s a) = \sum_{m = 1}^{M} \sum_{n = 1}^{N} \sum_{(a), (b)} \Phi_{m,n}(b_{(1)} a_s a_{(1)}) \overline{\Psi_{m,n}(a_s S(a_s^{-1})^* S(b_{(2)})^* S(a_s)^* S(a_{(2)})^*)}.
\end{equation*}
Since $a_s$ is invertible and a group like element it follows from antipode axiom for the Hopf algebra that $S(a_s) = a_s^{-1}$.
Because $a_s$ is self adjoint we have $S(a_s)^* = S(a_s^*) = a_s^{-1}$.
For the right coideal $\mathcal{B}$ we have, $b_{(1)} \in \mathcal{B}$.
This yields
\begin{equation*}
\begin{split}
\tau(a_s^{-1} b a_s a) 
  &= \sum_{m, k = 1}^M \sum_{n = 1}^N \sum_{(a), (b)} t^{\mathcal{B}}_{m,k}(b_{(1)}) \Phi_{k,n}(a_s a_{(1)}) \overline{\Psi_{m,n}(a_s^2 S(b_{(2)})^* a_s^{-1} S(a_{(2)})^*)} \\
  &= \sum_{k = 1}^M \sum_{n = 1}^N \sum_{(a)} \Phi_{k,n}(a_s a_{(1)}) \overline{\Psi_{k,n}\left(\sum_{(b)} b_{(1)}^* a_s^2 S(b_{(2)})^* a_s^{-1} S(a_{(2)})^*\right)},
\end{split}
\end{equation*}
where we used the spherical property of $\Psi$, which is $\Psi_{m,n}(b_{(1)} a_s^{-1} a_{(1)}) = \sum_{k = 1}^{M} t^{\mathcal{B}}_{m,k}(b_{(1)}) \Psi_{k,n}(a_s^{-1} a_{(1)})$, and the unitary of the representation, i.e. $t^{\mathcal{B}}_{m,k}(b_{(1)}) = \overline{t^{\mathcal{B}}_{k,m}(b_{(1)}^*)}$.
Note that $b_{(2)} \in A$ and by the property of $a_s$ we have $a_s^2 S(b_{(2)}^*) = S(b_{(2)})^* a_s^2$, therefore
\begin{equation*}
\sum_{(b)} b_{(1)}^* a_s^2 S(b_{(2)})^* a_s^{-1} S(a_{(2)})^* 
  = \sum_{(b)} b_{(1)}^* S(b_{(2)}^*) a_s S(a_{(2)})^*
  = \epsilon(b^*) a_s S(a_{(2)})^*.
\end{equation*}
Use that $\epsilon(b^*) = \overline{\epsilon(b)}$ and $\epsilon(a_s) = 1 = \epsilon(a_s^{-1})$, hence $\epsilon(b^*) = \overline{\epsilon(a_s b a_s^{-1})}$.
Combining these facts we have $\tau_{a_s}(a_s b a_s^{-1} a) = \epsilon(a_s b a_s^{-1}) \tau_{a_s}(a)$.
\end{proof}

In Theorem \ref{q-radial_thm_bi_invariant_weight} we need $a_s \in A_{\textrm{grp}}$ self adjoint such that $a_s^2 S(a^*) = S(a)^* a_s^2$ for all $a \in A$.
In the next lemma we show that each quantized function algebra $\Uq(\mathfrak{g})$ contains such an element.

\begin{lemma} \label{q-radial_lem_as}
Suppose $\mathfrak{g}$ is a semi-simple Lie algebra and suppose $\mathcal{U}_q(\mathfrak{g})$ is equipped with compact real form.
Let $\rho$ be the half sum of the positive roots and let $s_a = K_{-\rho} \in \check{\mathcal{U}}_q(\mathfrak{g})$.
Then $a_s$ is self adjoint and $a_s^2 S(a^*) = S(a)^* a_s^2$ for all $a \in \Uq(\mathfrak{g})$.
\end{lemma}

\begin{proof}
By Remark \ref{q-radial_rmk_star} the compact real form act trivially on $U^0$, hence $K_{-\rho}^* = K_{-\rho}$.

On the generators of $\mathcal{U}_q(\mathfrak{g})$ we check readily that $S^2(a) = K_{2\rho} a K_{-2\rho}$ for all $a \in \mathcal{U}_q(\mathfrak{g})$, see also \cite[Chapter 6, Proposition 6]{KS97} and \cite[p. 79, Exercise 4.1.1]{KS98}.
By \cite[Chapter 1, Proposition 10]{KS97} we have $S^{-1} = * \circ S \circ *$.
Apply $S^{-1} = * \circ S \circ *$ on both sides of $S^2(a) = K_{2\rho} a K_{-2\rho}$ to obtain $S(a) = (S(K_{2\rho}^*) S(a^*) S(K_{-2\rho}^*))^*$.
Since $K_{2\rho}$ and $K_{-2\rho}$ are self adjoint we have $S(a)^* = K_{-2\rho} S(a^*) K_{2\rho}$, from which the statement follows.
\end{proof}

\begin{example}
Let $n$ be an arbitrary positive integer and let $\mathfrak{g} = \mathfrak{sl}_n$ be the simple Lie algebra of type $A_n$.
The half sum of positive roots $\rho$ is given by $2\rho = \sum_{i = 1}^n (n-i+1)i \alpha_i$, see \cite[p. 684]{Kna02}.
From $\rho$ the elements $a_s$ of Lemma \ref{q-radial_lem_as} is given by
\begin{equation*}
a_s = K_{-\rho} = \prod_{i = 1}^{n} K_{i}^{-\frac{1}{2}(i - n - 1)i}.
\end{equation*}
For example $a_s$ for $n=1$, $n=2$ and $n=3$ is respectively equal to
\begin{equation*}
K^{-\frac{1}{2}}, \quad (K_1 K_2)^{-1}, \quad (K_1^3 K_2^4 K_3^3)^{-\frac{1}{2}}.
\end{equation*}
\end{example}

\section{Calculations of the radial part} \label{q-radial_sec_applications}

In this section we compute the radial part of the Casimir elements of the quantum analogues of $(\SU(2), \textrm{U}(1))$, $(\SU(2) \times \SU(2), \textrm{diag})$ and $(\SU(3), \textrm{U}(2))$.
We assume that every right coideal $\mathcal{B}_{\textbf{c}, \textbf{s}}$ is $*$-invariant for the compact real form on $\Uq(\mathfrak{g})$ and that $\textbf{c} \in \mathcal{C} \cap (\RR^{\times})^{k}$ and $\textbf{s} \in \mathcal{S} \cap \RR^k$.
However, the calculations in this chapter can be executed for general $\textbf{c}$ and $\textbf{s}$, although these cases will not have a nice limit if $q \to 1$ and often the radial part of the Casimir elements will not generate a second order $q$-difference equation of Askey-Wilson type, \cite{KLS10}.

Example \ref{q-radial_exm_radial_uqsl2} extends the results of \cite{Koo93} and gives rise to an alternative proof for \cite[Theorem 7.6]{Koe96} of which we skip the details.
Example \ref{q-radial_exm_radial_uqsl2uqsl2} coincides with \cite[Proposition 5.10]{AKR15}.
Example \ref{q-radial_exm_radial_uqsl3} extends Letzter's classification \cite[Theorem 8.2]{Let04} to an example for a quantum symmetric pair where the restricted root system is non-reduced.
We show that the radial part of the center in Example \ref{q-radial_exm_radial_uqsl3} restricted to the trivial representation coincides with \cite[Theorem 5.4]{DN98}.
Moreover we compute the radial part of the center in Example \ref{q-radial_exm_radial_uqsl3} in general, which has not yet been done before and which extend the result of Dijkhuizen and Noumi \cite{DN98} for the quantum analogue of $(\SU(3), \mathrm{U}(2))$ to the matrix valued case.
Note that the method described here is not restricted to the center, but this method can be used to calculate the radial part for any element of a quantum symmetric pair with simple generators.

\begin{example} \label{q-radial_exm_radial_uqsl2}
The spherical functions on quantum analogue of $(\SU(2), \mathrm{U}(1))$, see Example \ref{q-radial_exm_koornwinder1}, were first studied by Koornwinder \cite{Koo93}.
Let $\mathfrak{g} = \mathfrak{su}_2$.
Koornwinder \cite{Koo93} computed the radial part of the Casimir element $\Omega$ which is the generator of the center of $\mathcal{U}_q(\mathfrak{g})$.
In this example we show that the radial part of Koornwinder \cite{Koo93} coincides with the radial part $\Pi(\Omega) \in \check{\mathcal{B}}_{c, s} \tensor \check{\mathcal{A}} \tensor \check{\mathcal{B}}_{d, t}$ for well chosen values for $c, d, s$ and $t$.

Recall the definition of the rank $1$ coideal $\mathcal{B}_{c,s}$ from Example \ref{q-radial_exm_koornwinder1}.
Since $\mathcal{B}_{c,s}$ is $*$-invariant for the compact real form on $\Uq(\mathfrak{g})$ if follows from Remark \ref{q-radial_rmk_star} that $c = -1$.
Therefore we write $\mathcal{B}_{s} = \mathcal{B}_{-1,s}$.
We slightly modify the generators of $\mathcal{B}_{s}$ and write
\begin{equation*}
B_{s} = F + EK^{-1} + s(K^{-1} - 1),
\end{equation*}
with $s \in \mathcal{S} \cap \RR$.
Note that the right coideal $\mathcal{B}_{s}$ is also generated by $B_{s}$.
The Casimir element generating the center of $\mathcal{U}_q(\mathfrak{sl}_2)$ is given by
\begin{equation*}
\Omega = \frac{q^{-1} K + q K^{-1} - 2}{(q - q^{-1})^2} + EF,
\end{equation*}
see \cite[Proposition 3.2]{KS98}.
Fix $s, t \in \mathcal{S} \cap \RR$.
The quantum torus algebra $\check{\mathcal{A}}$ appearing in the quantum Iwasawa decomposition, Theorem \ref{q-radial_thm_quantumiwasawa}, is generated by $K^{\lambda}$ where $\lambda \in \frac{1}{2}\ZZ$.
For $\lambda \in \frac{1}{2}\ZZ$ we compute $A^{\lambda} \Omega$ as an element of $\check{\mathcal{B}}_{s} \tensor \check{\mathcal{A}} \tensor \check{\mathcal{B}}_{t}$.

First apply Theorem \ref{q-radial_thm_quantumiwasawa} to bring $K^{\lambda} EF$ in the quantum Iwasawa decomposition form.
We have
\begin{equation} \label{q-radial_eqn_koornwinder_iwasawa}
K^{\lambda} EF 
  = q^{2\lambda} E K^{\lambda} F
  = -q^{4\lambda} K^{\lambda + 1} F^2 - sq^{2\lambda}(K^{-1} - 1) K^{\lambda+1}F + q^{2\lambda} B_{s} K^{\lambda+1} F
\end{equation}
According to the proof of Theorem \ref{q-radial_thm_main} we proceed to decompose $K^{\lambda} F$ and $K^{\lambda} F^2$ inductively in $\check{\mathcal{B}}_{s} \check{\mathcal{A}} \check{\mathcal{B}}_{t}$.
\begin{equation*}
\begin{split}
K^{\lambda} F 
  &= K^{\lambda} (-EK^{-1} - t(K^{-1} - 1) + B_{t}) \\
  &= q^{2\lambda} (F + s(K^{-1} - 1) - B_{s}) K^{\lambda} + K^{\lambda}B_{t} - t(K^{\lambda-1} - K^{\lambda}) \\
  &= q^{4\lambda} K^{\lambda} F + \left(sq^{2\lambda} - t \right)(K^{\lambda-1} - K^{\lambda}) - q^{2\lambda} B_{s} + K^{\lambda} B_{t}.
\end{split}
\end{equation*}
Hence we have
\begin{equation}
\left(1 - q^{4\lambda}\right) K^{\lambda} F 
  = \left(sq^{2\lambda} - t\right)(K^{\lambda-1} - K^{\lambda}) - q^{2\lambda} B_{s} K^{\lambda} + K^{\lambda} B_{t}. \label{q-radial_eqn_exm_koo1}
\end{equation}
We now proceed one level higher and obtain
\begin{equation*}
\begin{split}
K^{\lambda} F^2 
  &= K^{\lambda}F(-EK^{-1} - t(K^{-1}-1) + B_{t})  \\
  &= -q^{2\lambda-2}EK^{-1}K^{\lambda}F + \frac{1}{(q-q^{-1})}(K^{\lambda} - K^{\lambda-2}) - \frac{t}{q^2}K^{\lambda-1}F + tK^{\lambda}F + K^{\lambda}FB_{t}  \\
  &= q^{4\lambda-2}K^{\lambda}F^2 + \left(t - q^{2\lambda-2}\right)K^{\lambda}F + \left(q^{2\lambda-2}s - \frac{t}{q^2}\right)K^{\lambda-1}F \\
  &\quad + \frac{1}{(q-q^{-1})}(K^{\lambda} - K^{\lambda-2}) - q^{2\lambda-2}B_{s}K^{\lambda}F + K^{\lambda}FB_{t},
\end{split}
\end{equation*}
hence we have
\begin{equation}
\label{q-radial_eqn_exm_koo2}
\begin{split}
\left(1 - q^{4\lambda-2}\right) K^{\lambda} F^2 
  &= \left(t - q^{2\lambda-2}\right)K^{\lambda}F + \left(q^{2\lambda-2}s - \frac{t}{q^2}\right)K^{\lambda-1}F \\
  &\quad - \frac{1}{(q-q^{-1})}(K^{\lambda} - K^{\lambda-2}) - q^{2\lambda-2}B_{s}K^{\lambda}F + K^{\lambda}FB_{t}.
\end{split}
\end{equation}
Assume that $q^{4\lambda} \neq 1$ and $q^{4\lambda-2} \neq 1$, which definitely holds if $K^{\lambda} \in \mathcal{A}_{\text{reg}}(EF)$.
Use \eqref{q-radial_eqn_exm_koo2} in \eqref{q-radial_eqn_koornwinder_iwasawa} to obtain
\begin{equation*}
\begin{split}
K^{\lambda}EF 
  &= - \frac{(s - q^{2\lambda+2}t)}{(1 - q^{4\lambda+2})} K^{\lambda+1} F
    - \frac{(s - q^{2\lambda} t)}{(1 - q^{4\lambda+2})} K^{\lambda} F
    + q^{2\lambda} \frac{1}{(1 - q^{4\lambda+2})} B_{s} K^{\lambda+1} F \\
  &\quad - q^{4\lambda+2} \frac{1}{(1 - q^{4\lambda+2})} K^{\lambda+1} F B_{t}
    - q^{4\lambda+2} \frac{K^{\lambda+1} - K^{\lambda}}{(q - q^{-1})(1 - q^{4\lambda+2})},
\end{split}
\end{equation*}
substitute \eqref{q-radial_eqn_exm_koo1} gives that $K^{\lambda} EF$ is equal to
\begin{equation*}
\begin{split}
&q^{2\lambda} \frac{
  (t - q^{2\lambda+2}s)(s - q^{2\lambda+2}t)
}{
  (1 - q^{4\lambda+2})(1 - q^{4\lambda+4})
} (K^{\lambda+1} - K^{\lambda})
- q^{2\lambda} \frac{
  (t - q^{2\lambda}s)(s - q^{2\lambda}t)
}{
  (1 - q^{4\lambda})(1 - q^{4\lambda+2})
} (K^{\lambda} - K^{\lambda-1}) \\
&\quad + q^{4\lambda+2} \frac{
  K^{\lambda+1} - K^{\lambda-1}
}{
  (1 - q^{4\lambda+2})(q - q^{-1})
} 
- q^{2\lambda} \frac{
  (s - q^{2\lambda} t - q^{2\lambda+2} t + q^{4\lambda+2} s)
}{
  (1 - q^{4\lambda})(1 - q^{4\lambda+4})
} K^{\lambda} B_{t} \\
&\quad - q^{2\lambda} \frac{
  (t - q^{2\lambda}s - q^{2\lambda+2}s + q^{4\lambda+2}t)
}{
  (1 - q^{4\lambda})(1 - q^{4\lambda+4})
} B_{s} K^{\lambda}
+ q^{2\lambda} \frac{
  (s - 2q^{2\lambda+2}t + q^{4\lambda+4} s)
}{
  (1 - q^{4\lambda+2})(1 - q^{4\lambda+4})
} K^{\lambda+1} B_{t} \\
&\quad + q^{2\lambda} \frac{
  (t - 2q^{2\lambda+2}s - q^{4\lambda+4}t)
}{
  (1 - q^{4\lambda+2})(1 - q^{4\lambda+4})
} B_{s} K^{\lambda+1}
+ q^{2\lambda} \frac{
  (1 + q^{4\lambda+4})
}{
  (1 - q^{4\lambda+2})(1 - q^{4\lambda+4})
} B_{s} K^{\lambda+1} B_{t} \\
&\quad - q^{4\lambda+2} \frac{
  1
}{
  (1 - q^{4\lambda+2})(1 - q^{4\lambda+4})
} B_{s}^2 K^{\lambda+1}
- q^{4\lambda+2} \frac{
  1
}{
  (1 - q^{4\lambda+2})(1 - q^{4\lambda+4})
} K^{\lambda+1} B_{t}^2.
\end{split}
\end{equation*}
Add $K^{\lambda}(q^{-1}K + qK^{-1} - 2)$ to both sides so that we have $K^{\lambda} \Omega \in \check{\mathcal{B}}_{s} \check{\mathcal{A}} \check{\mathcal{B}}_{t}$.

Take $\sigma, \tau \in \RR$ such that $s = q^{\frac{1}{2}}(q^{-\sigma} - q^{\sigma})(q^{-1} - q)^{-1}$, $t = q^{\frac{1}{2}}(q^{-\tau} - q^{\tau})(q^{-1} - q)^{-1}$.
For $\lambda \neq 0, -\frac{1}{2}, -1$ we obtain
\begin{equation}
\label{q-radial_eqn_koornwinder_qdiff}
\begin{split}
q(q-q^{-1})^2 K^{\lambda} \Omega 
  &= f(q^{\lambda}) (K^{\lambda+1} - K^{\lambda}) + f(q^{-\lambda-1}) (K^{\lambda} - K^{\lambda-1}) \\
  & \quad + (1 - q)^2 K^{\lambda} + (\check{\mathcal{B}}_{c,s})_+ \check{\mathcal{A}} (\check{\mathcal{B}}_{d,t})_+
\end{split}
\end{equation}
where for $C \subseteq \check{\mathcal{U}}_q(\mathfrak{g}')$ we use the notation $C_+ = \{ c \in C : \epsilon(c) = 0 \}$ and
\begin{equation*}
\begin{split}
f(q^{\lambda}) &= \frac{
  (1 + q^{2\lambda + \sigma + \tau + 2})(1 + q^{2\lambda -\sigma - \tau + 2})(1 - q^{2\lambda + \sigma - \tau + 2})(1 - q^{2\lambda -\sigma + \tau + 2})
}{
  (1 - q^{4\lambda + 2})(1 - q^{4\lambda + 4})
}.
\end{split}
\end{equation*}
Identifying $z$ with $q^{2\lambda+1}$ and $A^{\lambda-1}, A^{\lambda}, A^{\lambda_1}$ with $q^{-1}z, z, qz$ in \eqref{q-radial_eqn_koornwinder_qdiff} we find the second order $q$-difference operator for Askey-Wilson polynomials with two free parameters.
This observation is a key ingredient, see \cite[Lemma 5.1]{Koo93}, of some of the main results in Koornwinder \cite[Theorem 5.2 and Theorem 5.3]{Koo93}.

In a similar fashion the $\check{\mathcal{B}}_{\textbf{c}, \textbf{s}} \check{\mathcal{A}} \check{\mathcal{B}}_{\textbf{d}, \textbf{t}}$-decomposition of the Casimir element gives rise to the $q$-difference equation for the Askey-Wilson polynomials in four free parameters described in the second alternative proof of \cite[Remark 7.7]{Koe96}.
Hence we can obtain an alternative proof for \cite[Theorem 7.6]{Koe96} which requires conjugation of the radial part of the Casimir operator.
We skip the details.
\end{example}

\begin{example} \label{q-radial_exm_radial_uqsl2uqsl2}
Recall that the right coideal $\mathcal{B}$ of the quantum analogue of $(\SU(2) \times \SU(2), \diag)$ is generated by
\begin{equation*}
B_1 = F_1 - E_2 K_1^{-1}, \quad B_2 = F_2 - E_1 K_2^{-1}, \quad K^{\pm 1} = (K_1 K_2^{-1})^{\pm 1},
\end{equation*}
see also Example \ref{q-radial_exm_A1xA1}.
Let $\mathfrak{g} = \mathfrak{su}_2 \oplus \mathfrak{su}_2$.
The Casimir elements generating the center of $\Uq(\mathfrak{su}_2)$ are given by
\begin{equation*}
\begin{split}
\Omega_1 = \frac{q^{-1} K_1 + q K_1^{-1} - 2}{(q - q^{-1})^2} + E_1 F_1, \quad
\Omega_2 = \frac{q^{-1} K_2 + q K_2^{-1} - 2}{(q - q^{-1})^2} + E_2 F_2.
\end{split}
\end{equation*}
Note that $\sigma : \Uq(\mathfrak{g}) \to \Uq(\mathfrak{g}) : X_i \mapsto X_{\tau(i)}$ where $X = E_i, F_i, K_i$ and $\tau = (1\,2)$ is a Hopf algebra isomorphism, i.e. $\sigma$ is the flip operator.
Therefore it is sufficient to calculate the radial part only for $\Omega_1$.

The quantum torus $\mathcal{A}$ of the quantum Iwasawa decomposition, Theorem \ref{q-radial_thm_quantumiwasawa}, of $\Uq(\mathfrak{g})$ is generated by $A^{\pm 1} = (K_1 K_2)^{\pm 1}$.
We compute the radial part of $A^{\lambda} \Omega_1$ in $\check{\mathcal{B}} \check{\mathcal{A}} \check{\mathcal{B}}$ where $\lambda \in \frac{1}{2}\ZZ$.
We apply the quantum Iwasawa decomposition, Theorem \ref{q-radial_thm_quantumiwasawa}, on $A^{\lambda}\Omega_1$ and obtain
\begin{equation}
\label{q-radial_eqn_sl2sl2.1}
\begin{split} 
  A^{\lambda} E_1 F_1 = q^{2\lambda} E_1 A^{\lambda} F_1
  &= q^{2\lambda} (F_2 K_2 - B_2 K_2) A^{\lambda} F_1  \\
  &= q^{4\lambda + 2} K^{-\frac{1}{2}} K^{\lambda+\frac{1}{2}} F_1 F_2 - q^{2\lambda} B_2 K^{-\frac{1}{2}} K^{\lambda + \frac{1}{2}} F_1.
\end{split}
\end{equation}
Inductively we determine the $\mathcal{B} \mathcal{A} \mathcal{B}$-decomposition of $A^{\lambda} F_1$ and $A^{\lambda} F_1^2$.
\begin{equation*}
A^{\lambda} F_1 = A^{\lambda} (B_1 + E_2 K_1^{-1})
  = A^{\lambda} B_1 + q^{2\lambda} E_2 K_1^{-1} A^{\lambda}
  = A^{\lambda} B_1 + q^{4\lambda} A^{\lambda} F_1 - q^{2\lambda} B_1 A^{\lambda}.
\end{equation*}
Assume that $\lambda \neq 0$, then
\begin{equation} 
\label{q-radial_eqn_sl2sl2.2}
A^{\lambda} F_1 = \frac{1}{1 - q^{4\lambda}} \left(
  A^{\lambda} B_1 - q^{2\lambda} B_1 A^{\lambda}
\right).
\end{equation}
For $A^{\lambda} F_1 F_2$ we have
\begin{equation} 
\label{q-radial_eqn_sl2sl2.3}
\begin{split}
A^{\lambda} F_1 F_2 
  &= A^{\lambda} F_1 (B_2 + E_1 K_2^{-1})  \\
  &= A^{\lambda} F_1 B_2 + A^{\lambda} \left( E_1 F_1 - \frac{K_1 - K_1^{-1}}{q - q^{-1}} \right) K_2^{-1}  \\
  &= A^{\lambda} F_1 B_2 + q^{2\lambda} E_1 K_2^{-1} A^{\lambda} F_1 - A^{\lambda} \left( \frac{K_1 - K_1^{-1}}{q - q^{-1}} \right) K_2^{-1}  \\
  &= A^{\lambda} F_1 B_2 + q^{4\lambda} A^{\lambda} F_1 F_2 - q^{2\lambda} B_2 A^{\lambda} F_1 - A^{\lambda} \left( \frac{K_1 - K_1^{-1}}{q - q^{-1}} \right) K_2^{-1}.
\end{split}
\end{equation}
Assume $\lambda \neq -\frac{1}{2}$ and substitute (\ref{q-radial_eqn_sl2sl2.2}) and (\ref{q-radial_eqn_sl2sl2.3}) into (\ref{q-radial_eqn_sl2sl2.1}) then gives
\begin{equation*}
\begin{split}
A^{\lambda} E_1 F_2 
  &= - \frac{q^{4\lambda + 2}}{(1 - q^{4\lambda+2})(q - q^{-1})} K^{\frac{1}{2}} A^{\lambda + \frac{1}{2}}
    + \frac{q^{4\lambda + 2}}{(1 - q^{4\lambda+2})(q - q^{-1})} K^{-\frac{1}{2}} A^{\lambda - \frac{1}{2}} \\
    &\qquad + \frac{q^{4\lambda + 2}}{(1 - q^{4\lambda+2})^2} K^{-\frac{1}{2}} A^{\lambda+\frac{1}{2}} B_1 B_2
    - \frac{q^{6\lambda + 3}}{(1 - q^{4\lambda+2})^2} K^{-\frac{1}{2}} B_1 A^{\lambda+\frac{1}{2}} B_2 \\
    &\qquad - \frac{q^{2\lambda + 1}}{(1 - q^{4\lambda+2})^2} K^{-\frac{1}{2}} B_2 A^{\lambda+\frac{1}{2}} B_1
    + \frac{q^{4\lambda + 2}}{(1 - q^{4\lambda+2})^2} K^{-\frac{1}{2}} B_2 B_1 A^{\lambda+\frac{1}{2}}.
\end{split}
\end{equation*}
Hence for $\lambda \neq \frac{1}{2}$ the radial part of the Casimir becomes
\begin{equation*}
\begin{split}
A^{\lambda} \Omega_1
  &= \frac{1}{q} \frac{(1 - q^{4\lambda + 4})}{(q - q^{-1})^2 (1 - q^{4\lambda + 2})} K^{\frac{1}{2}} A^{\lambda + \frac{1}{2}}
    + q \frac{(1 - q^{4\lambda})}{(q - q^{-1})^2 (1 - q^{4\lambda + 2})} K^{-\frac{1}{2}} A^{\lambda - \frac{1}{2}}
    - \frac{2}{(q - q^{-1})^2} A^{\lambda} \\
  &\qquad + \frac{q^{4\lambda + 2}}{(1 - q^{4\lambda+2})^2} K^{-\frac{1}{2}} A^{\lambda+\frac{1}{2}} B_1 B_2
    - \frac{q^{6\lambda + 3}}{(1 - q^{4\lambda+2})^2} K^{-\frac{1}{2}} B_1 A^{\lambda+\frac{1}{2}} B_2 \\
  &\qquad - \frac{q^{2\lambda + 1}}{(1 - q^{4\lambda+2})^2} K^{-\frac{1}{2}} B_2 A^{\lambda+\frac{1}{2}} B_1
    + \frac{q^{4\lambda + 2}}{(1 - q^{4\lambda+2})^2} K^{-\frac{1}{2}} B_2 B_1 A^{\lambda+\frac{1}{2}}.
\end{split}
\end{equation*}
The Hopf-algebra isomorphism $\sigma$ maps $\sigma(K) = K^{-1}$ and $\sigma(B_1) = B_2$.
Since $\sigma(\Omega_1) = \Omega_2$ it follows that
\begin{equation*}
\begin{split}
A^{\lambda} \Omega_2
  &= \frac{1}{q} \frac{(1 - q^{4\lambda + 4})}{(q - q^{-1})^2 (1 - q^{4\lambda + 2})} K^{-\frac{1}{2}} A^{\lambda + \frac{1}{2}}
    + q \frac{(1 - q^{4\lambda})}{(q - q^{-1})^2 (1 - q^{4\lambda + 2})} K^{\frac{1}{2}} A^{\lambda - \frac{1}{2}}
    - \frac{2}{(q - q^{-1})^2} A^{\lambda} \\
  &\qquad + \frac{q^{4\lambda + 2}}{(1 - q^{4\lambda+2})^2} K^{\frac{1}{2}} A^{\lambda+\frac{1}{2}} B_2 B_1
    - \frac{q^{6\lambda + 3}}{(1 - q^{4\lambda+2})^2} K^{\frac{1}{2}} B_2 A^{\lambda+\frac{1}{2}} B_1 \\
  &\qquad - \frac{q^{2\lambda + 1}}{(1 - q^{4\lambda+2})^2} K^{\frac{1}{2}} B_1 A^{\lambda+\frac{1}{2}} B_2
    + \frac{q^{4\lambda + 2}}{(1 - q^{4\lambda+2})^2} K^{\frac{1}{2}} B_1 B_2 A^{\lambda+\frac{1}{2}}.
\end{split}
\end{equation*}
The radial parts computed above correspond to the radial parts of \cite[Proposition 5.10]{AKR15}.
For the special case of the counit representation, the maps $\Pi_{\epsilon, \epsilon}(\Omega_1)$ and $\Pi_{\epsilon, \epsilon}(\Omega_2)$ coincide.
Identifying $z = q^{2\lambda + 1}$ and $A^{\lambda-\frac{1}{2}}$, $A^{\lambda}$, $A^{\lambda+\frac{1}{2}}$ with $q^{-1}z$, $z$, $qz$, we obtain the second order $q$-difference equation for the Chebyshev polynomials of the second kind.
In \cite{AKR15} we continue studying matrix valued spherical functions that are solutions to the second order $q$-difference equations $\Pi_{t, t}(\Omega_i)$ for all irreducible finite dimensional representations $t$.
\end{example}

\begin{example} \label{q-radial_exm_radial_uqsl3}
In this last example we study the radial part of the center of the quantum analogue of $(\SU(3), \mathrm{U}(2))$, see also Example \ref{q-radial_exm_uqsl3}.
This case is excluded in Letzter \cite{Let04}, because the restricted root system is non-reduced.
The radial part of the center of the quantum analogue of $(\SU(3), \mathrm{U}(2))$ restricted to the trivial representation is identified with the results of Dijkhuizen and Noumi \cite[Theorem 5.4]{DN98}.
Recall that the right coideal $\mathcal{B}_{\textbf{c}}$ is generated by
\begin{equation*}
B_1^{\textbf{c}} = F_1 - c_1 E_1 K_1^{-1}, \quad B_2^{\textbf{c}} = F_2 - c_2 E_2 K_1^{-1}, \quad K^{\pm 1} = (K_1 K_2^{-1})^{\pm 1}.
\end{equation*}
Let $\mathfrak{g} = \mathfrak{su}_3$.
If we allow third roots of $K_1$ and $K_2$ in $\Uq(\mathfrak{g})$ the center of $\mathcal{B}_{c}$ is formally generated by the two second order Casimir elements of the form
\begin{equation} \label{q-radial_eqn_CasimirUqsl3}
\begin{split}
\Omega_1 &= K_1^{\frac{1}{3}} K_2^{-\frac{1}{3}} \left[
  \frac{
    q^{-2} K_1 K_2 + K_1^{-1} K_2 + q^2 K_1^{-1} K_2^{-1}
  }{
    (q - q^{-1})^2
  }
  +
  q K_2 E_1 F_1 + q^{-1} K_1^{-1} E_2 F_2 - \check{E_3} F_3
\right], \\
\Omega_2 &= K_1^{\frac{1}{3}} K_2^{-\frac{1}{3}} \left[
  \frac{
    q^{-2} K_1 K_2 + K_1 K_2^{-1} + q^2 K_1^{-1} K_2^{-1}
  }{
    (q - q^{-1})^2
  }
  +
  q^{-1} K_2^{-1} E_1 F_1 + q K_1 E_2 F_2 - E_3 \check{F_3}
\right],
\end{split}
\end{equation}
where $E_3 = E_1 E_2 - q E_2 E_1 = [E_1, E_2]_q$, $F_3 = F_1 F_2 - q F_2 F_1 = [F_1, F_2]_q$, $\check{E_3} = E_1 E_2 - q^{-1} E_2 E_1 = [E_1, E_2]_{q^{-1}}$ and $\check{F_3} = F_1 F_2 - q^{-1} F_2 F_1 = [F_1, F_2]_{q^{-1}}$.
See also \cite{GZ91} and \cite{Rod91} for the explicit expression of \eqref{q-radial_eqn_CasimirUqsl3}.
We compute the radial part for $\widetilde{\Omega}_1 = K_1^{-\frac{1}{3}} K_2^{\frac{1}{3}} \Omega_1 \in \mathcal{U}_q(\mathfrak{g})$ and $\widetilde{\Omega}_2 = K_1^{-\frac{1}{3}} K_2^{\frac{1}{3}} \Omega_2 \in \mathcal{U}_q(\mathfrak{g})$. 
The computations modulo $(\mathcal{B}_{\textbf{c}})_+$ on the left and $(\mathcal{B}_{\textbf{d}})_+$ on the right are identified explicitly with orthogonal polynomials of the $q$-Askey scheme \cite{KLS10}.

The computation of the radial part is tedious but straightforward.
Therefore we omit most of the computations and use a couple of tricks to simplify the computations.
Introduce the flip operator $\sigma$ on $\Uq(\mathfrak{g})$ defined on the generators by $\sigma(X_i) = \sigma(X_{\tau(i)})$ and $\sigma(c_i) = \sigma(c_{\tau(i)})$ where $\tau = (1\ 2)$.
Note that the flip operator $\sigma$ is a Hopf $*$-algebra isomorphism on $\Uq(\mathfrak{g})$ leaving $\mathcal{B}_{\textbf{c}}$ and $\mathcal{B}_{\textbf{d}}$ invariant so that $\sigma(\widetilde{\Omega}_1) = \widetilde{\Omega}_2$.
Hence we only have to compute the radial part of $\widetilde{\Omega}_1$.

The quantum torus $\mathcal{A}$ is generated by $A^{\pm 1} = (K_1 K_2)^{\pm 1}$.
Let $\lambda \in \frac{1}{2}\ZZ$.
We first compute the radial part of expressions $A^{\lambda} F_X$ up to degree $4$, i.e. $|X| \leq 4$, where $X = (1)$, $X = (1, 1)$, $X = (1, 2)$, $X = (1, 1, 2)$, $X = (1, 2, 1)$, $X = (2, 2, 1)$, $X = (1, 2, 1, 2)$, $X = (1, 1, 2, 2)$ and $X = (2, 1, 1, 2)$.
Other expressions $A^{\lambda} F_X$ needed for the computation of the radial part can be obtained from the flip operator $\sigma$.
By Theorem \ref{q-radial_thm_main} we have
\begin{equation} \label{q-radial_eqn_uqsl3_X1}
A^{\lambda} F_1 = A^{\lambda} B_1^{\textbf{d}} - \frac{d_1}{c_1} q^{\lambda} B_1^{\textbf{c}} A^{\lambda} + \frac{d_1}{c_1} q^{2\lambda} A^{\lambda} F_1,
\end{equation}
Next, by Theorem \ref{q-radial_thm_main}, we have permutations $F_1 F_2 \rightarrow F_2 F_1 \xrightarrow{\sigma} F_1 F_2$ and $F_1^2 \rightarrow F_1^2$, where $\xrightarrow{\sigma}$ is used for the flip operator, 
\begin{equation}
\label{q-radial_eqn_uqsl3_X2}
\begin{split}
A^{\lambda} F_1 F_2 &= A^{\lambda} F_1 B_2^{\textbf{d}} 
  - d_2 \left( \frac{K_1 K_2^{-1} - K_1^{-1} K_2^{-1}}{q - q^{-1}} \right) A^{\lambda}
  - \frac{d_2}{c_2} q^{\lambda + 1} B_2^{\textbf{c}} A^{\lambda} F_1
  + \frac{d_2}{c_2} q^{2\lambda + 1} A^{\lambda} F_2 F_1, \\
A^{\lambda} F_1^2 &= A^{\lambda} F_1 B_1^{\textbf{d}} - \frac{d_1}{c_1} q^{\lambda - 2} B_1^{\textbf{c}} A^{\lambda} F_1 + \frac{d_1}{c_1} q^{2\lambda - 2} A^{\lambda} F_1^2.
\end{split}
\end{equation}
Using the flip operator we can now compute $A^{\lambda} F_X \in \check{\mathcal{B}}_{\textbf{c}} \check{\mathcal{A}} \check{\mathcal{B}}_{\textbf{d}}$ for all $|X| \leq 2$.
We proceed with $|X| = 3$.
The expressions used of degree three have permutation $F_1^2 F_2 \rightarrow F_2 F_1^2 \rightarrow F_1 F_2 F_1 \rightarrow F_1^2 F_2$, therefore we calculate
\begin{equation}
\label{q-radial_eqn_uqsl3_X3}
\begin{split}
A^{\lambda} F_1^2 F_2 &= A^{\lambda} F_1^2 B_2^{\textbf{d}} 
  - d_2 \left( \frac{q^4 K - K_1^{-1} K_2^{-1}}{q - q^{-1}} \right) A^{\lambda} F_1^2
  - d_2 q \left( \frac{ K - K_1^{-1} K_2^{-1}}{q - q^{-1}} \right) A^{\lambda} F_1 \\
  &\quad - \frac{d_2}{c_2} q^{\lambda + 2} B_2^{\textbf{c}} A^{\lambda} F_1^2
    + \frac{d_2}{c_2} q^{2\lambda + 2} A^{\lambda} F_2 F_1^2, \\
A^{\lambda} F_2 F_1^2 &= A^{\lambda} F_2 F_1 B_1^{\textbf{d}}
  - d_1 q^{-2} \left( \frac{K^{-1} - K_1^{-1} K_2^{-1}}{q - q^{-1}} \right) A^{\lambda} F_1 \\
  &\qquad - \frac{d_1}{c_1} q^{\lambda-1} B_1^{\textbf{c}} A^{\lambda} F_2 F_1
  + \frac{d_1}{c_1} q^{2\lambda-1} A^{\lambda} F_1 F_2 F_1, \\
A^{\lambda} F_1 F_2 F_1 &= A^{\lambda} F_1 F_2 B_1^{\textbf{d}}
  - d_1 \left( \frac{K^{-1} - q^{2} K_1^{-1} K_2^{-1}}{q - q^{-1}} \right) A^{\lambda} F_1 \\
  &\qquad - \frac{d_1}{c_1} q^{\lambda-1} B_1^{\textbf{c}} A^{\lambda} F_1 F_2
  + \frac{d_1}{c_1} q^{2\lambda - 1} A^{\lambda} F_1^2 F_2.
\end{split}
\end{equation}
For $|X|=4$ the computation splits into two permutations $F_1 F_2 F_1 F_2 \rightarrow F_2 F_1 F_2 F_1 \xrightarrow{\sigma} F_1 F_2 F_1 F_2$ and $F_1^2 F_2^2 \rightarrow F_2 F_1^2 F_2 \rightarrow F_2^2 F_1^2 \xrightarrow{\sigma} F_1 F_2^2 F_1 \xrightarrow{\sigma} F_1^2 F_2^2$. 
By Theorem \ref{q-radial_thm_main} the first permutation is
\begin{equation}
\label{q-radial_eqn_uqsl3_X4.1}
\begin{split}
A^{\lambda} F_1 F_2 F_1 F_2 &= A^{\lambda} F_1 F_2 F_1 B_2^{\textbf{d}}
  - d_2 q^{-1} \left( \frac{K_1 K_2^{-1} - K_1^{-1} K_2^{-1}}{q - q^{-1}} \right) A^{\lambda} F_2 F_1 \\
  &\quad - d_2 \left( \frac{K_1 K_2^{-1} - q^{-2} K_1^{-1} K_2^{-1}}{q - q^{-1}} \right) A^{\lambda} F_1 F_2 \\
  &\quad - \frac{d_2}{c_2} q^{\lambda} B_2^{\textbf{c}} A^{\lambda} F_1 F_2 F_1
  + \frac{d_2}{c_2} q^{2\lambda} A^{\lambda} F_2 F_1 F_2 F_1
\end{split}
\end{equation}
and the second permutation is
\begin{equation}
\label{q-radial_eqn_uqsl3_X4.2}
\begin{split}
A^{\lambda} F_1^2 F_2^2 &= A^{\lambda} F_1^2 F_2 B_2^{\textbf{d}}
  - d_2 \left( \frac{(q + q^{-1}) K_1 K_2^{-1} - (q^{-3} + q^{-1}) K_1^{-1} K_2^{-1}}{q - q^{-1}} \right) A^{\lambda} F_1 F_2 \\
  &\quad - \frac{d_2}{c_2} q^{\lambda} B_2^{\textbf{c}} A^{\lambda} F_1^2 F_2
  + \frac{d_2}{c_2} q^{2\lambda} A^{\lambda} F_2 F_1^2 F_2, \\
A^{\lambda} F_2 F_1^2 F_2 &= A^{\lambda} F_2 F_1^2 B_2^{\textbf{d}}
  - d_2 \left( \frac{(1 + q^{-2}) K_1 K_2^{-1} - (1 + q^{-2}) K_1^{-1} K_2^{-1}}{q - q^{-1}} \right) A^{\lambda} F_2 F_1 \\
  &\quad - \frac{d_2}{c_2} q^{\lambda} B_2^{\textbf{c}} A^{\lambda} F_2 F_1^2
  + \frac{d_2}{c_2} q^{2\lambda} A^{\lambda} F_2^2 F_1^2.
\end{split}
\end{equation}

With Theorem \ref{q-radial_thm_quantumiwasawa} we bring all terms of $A^{\lambda}\widetilde{\Omega}_1$ in the form of the quantum Iwasawa decomposition, for $K_2 E_1 F_1$ and $K_1^{-1} E_2 F_2$ we have
\begin{align*}
A^{\lambda} K_2 E_1 F_1 &= \frac{q^{2\lambda+3}}{c_2} K^{-1} A^{\lambda+1} F_2 F_1 - \frac{q^{\lambda-1}}{c_2} B_2^{\textbf{c}} K^{-1} A^{\lambda+1} F_1, \\
A^{\lambda} K_1^{-1} E_2 F_2 &= \frac{q^{2\lambda+1}}{c_1} A^{\lambda} F_1 F_2 - \frac{q^{\lambda+1}}{c_1} B_1^{\textbf{c}} A^{\lambda} F_2.
\end{align*}
To find the quantum Iwasawa decomposition for $A^{\lambda} \check{E_3} F_3$ and $A^{\lambda} E_3 \check{F_3}$ we apply a trick.
The $q$-commutator of $B_1^\textbf{c}$ and $B_2^\textbf{c}$ is
\begin{align*}
[B_1^\textbf{c}, B_2^\textbf{c}]_q &= F_3 - c_1 c_2 q^{-1} \check{E_3} K_1^{-1} K_2^{-1}
  + c_1 (q - q^{-1}) E_2 F_2 K_1^{-1} \\
  &\quad - c_1 q \left( \frac{K_2 - K_2^{-1}}{q - q^{-1}} \right) K_1^{-1}
  + c_2 \left( \frac{K_1 - K_1^{-1}}{q - q^{-1}} \right) K_2^{-1}.
\end{align*}
Hence we have
\begin{equation}
\label{q-radial_eqn_uqsl3_commutanttrick}
\begin{split}
A^{\lambda} &\check{E_3} F_3 = q^{2\lambda} \check{E_3} A^{\lambda} F_3 \\
  &= \frac{q^{2\lambda+1}}{c_1c_2} \left( -[B_1^\textbf{c}, B_2^\textbf{c}]_q + F_3 + c_1(q - q^{-1})E_2 F_2 K_1^{-1} \right. \\
  &\qquad \qquad \left. - c_1 q \left(\frac{K_2 - K_2^{-1}}{q-q^{-1}}\right)K_1^{-1} + c_2\left(\frac{K_1-K_1^{-1}}{q-q^{-1}}\right)K_2^{-1} \right) A^{\lambda+1} F_3, \\
  &= \frac{q^{2\lambda+1}}{c_1 c_2} \left( -[B_1^{\textbf{c}}, B_2^{\textbf{c}}]_q A^{\lambda+1} F_3 + q^{2\lambda+2} A^{\lambda+1} F_3^2 + q^{2\lambda + 3}(q - q^{-1}) A^{\lambda+1} F_1 F_2 F_3 \right. \\
  &\qquad \left. -q^{\lambda+2}(q - q^{-1}) B_1^{\textbf{c}} A^{\lambda+1} F_2 F_3 + \frac{c_2 K_1 K_2^{-1} + (c_1 q - c_2) K_1^{-1} K_2^{-1} - c_1 q K_1^{-1} K_2}{(q - q^{-1})} A^{\lambda+1} F_3 \right), \\
\end{split}
\end{equation}

We assume that $\mathcal{B}_{\textbf{c}}$ and $\mathcal{B}_{\textbf{d}}$ are $*$-invariant with respect to the $*$-structure corresponding to the compact real form.
By Remark \ref{q-radial_rmk_star} we find that $c_1 c_2 = q^3$ and $d_1 d_2 = q^3$. 
Let $\sim$ be the equivalence relation module $(\mathcal{B}_{\textbf{c}})_+$ from the left and $(\mathcal{B}_{\textbf{d}})_+$ from the right.
From \eqref{q-radial_eqn_uqsl3_X1} and \eqref{q-radial_eqn_uqsl3_X3} it follows that $A^{\lambda}F_X \sim 0$ if $|X| = 1$ or $|X| = 3$.
For $|X| = 2$ we have from \eqref{q-radial_eqn_uqsl3_X2} that $A^{\lambda} F_1^2 \sim 0$ and 
\begin{equation}
\label{q-radial_eqn_uqsl3_restricted_X2}
\begin{split}
(q - q^{-1})(1 - q^{4\lambda + 2}) A^{\lambda} F_1 F_2 
  &\sim -\left(d_2 + \frac{q^{2\lambda + 4}}{c_2}\right) A^{\lambda}
    + \left(d_2 + \frac{q^{2\lambda + 4}}{c_2}\right) A^{\lambda-1},
\end{split}
\end{equation}
With \eqref{q-radial_eqn_uqsl3_X4.1} and \eqref{q-radial_eqn_uqsl3_X4.2} we compute
\begin{equation}
\label{q-radial_eqn_uqsl3_restricted_X4}
\begin{split}
(q - q^{-1})(1 - q^{4\lambda}) A^{\lambda} F_1 F_2 F_1 F_2
  &\sim -\left(d_2 q^{-1} + \frac{q^{2\lambda + 3}}{c_2}\right) A^{\lambda} F_2 F_1
    + \left(d_2 q^{-1} + \frac{q^{2\lambda + 1}}{c_2}\right) A^{\lambda-1} F_2 F_1 \\
  &\quad - \left(d_2 \mu + \frac{q^{2\lambda + 2}}{c_2}\right) A^{\lambda} F_1 F_2
    + \left(d_2 q^{-2} + \frac{q^{2\lambda + 2}}{c_2}\right) A^{\lambda-1} F_1 F_2, \\
(q - q^{-1}) (1 - q^{8\lambda}) A^{\lambda} F_2 F_1^2 F_2
  &\sim  - \left(d_2(1+q^{-2}) + \frac{q^{2\lambda + 3}}{c_2} (q + q^{-1})\right) A^{\lambda} F_2 F_1 \\
  &\quad + \left(d_2(1 + q^{-2}) + \frac{q^{2\lambda + 3}}{c_2} (q^{-3} + q^{-1})\right) A^{\lambda-1} F_2 F_1 \\
  &\quad - q^{4\lambda} \left(d_1(1 + q^{-2}) + \frac{q^{2\lambda + 3}}{c_1} (q + q^{-1})\right) A^{\lambda} F_1 F_2 \\
  &\quad + q^{4\lambda} \left(d_1 (1 + q^{-2}) + \frac{q^{2\lambda + 3}}{c_1} (q^{-3} + q^{-1})\right) A^{\lambda-1} F_1 F_2.
\end{split}
\end{equation}
From \eqref{q-radial_eqn_uqsl3_restricted_X2} and \eqref{q-radial_eqn_uqsl3_restricted_X4} it follows
\begin{align*}
A^{\lambda} F_1 F_2 &\sim f_{12}(\lambda) A^{\lambda} + g_{12}(\lambda) A^{\lambda - 1}, \\
A^{\lambda} F_1 F_2 F_1 F_2 &\sim f_{1212}(\lambda) A^{\lambda} + g_{1212}(\lambda) A^{\lambda-1} + h_{1212}(\lambda) A^{\lambda-2}, \\
A^{\lambda} F_2 F_1^2 F_2 &\sim f_{2112}(\lambda) A^{\lambda} + g_{2112}(\lambda) A^{\lambda-1} + h_{2112}(\lambda) A^{\lambda-2},
\end{align*}
where
\begin{align*}
f_{12}(\lambda) &= g_{12}(\lambda) = \frac{d_2 q}{(1 - q^2)} \frac{(1 + q^{2\lambda + 1})}{(1 - q^{4\lambda + 2})}, \\ 
f_{1212}(\lambda) &= \frac{1}{(1 - q^2)^2} \frac{
  (q^{2\lambda + 5} + d_2^2 q^2 + q^4 + \frac{d_2}{c_2} q^{2\lambda + 4})(1 + q^{2\lambda + 1})
}{
  (1 - q^{4\lambda + 2})(1 - q^{4\lambda})
}, \\
g_{1212}(\lambda) &= -(f_{1212}(\lambda) + h_{1212}(\lambda)), \\
h_{1212}(\lambda) &= \frac{1}{(1 - q^2)^2} \frac{
  (q^4 + q^{2\lambda+3} + d_2^2 + \frac{d_2}{c_2} q^{2\lambda + 4})(1 + q^{2\lambda - 1})
}{
  (1 - q^{4\lambda})(1 - q^{4\lambda - 2})
}, \\
f_{2112}(\lambda) &= q^3 \frac{(1 + q^2)}{(1 - q^2)^2} \frac{
  (1 + q^{2\lambda + 3})(1 + q^{2\lambda + 1})
}{
  (1 - q^{4\lambda + 2})(1 - q^{4\lambda})
}, \\
g_{2112}(\lambda) &= -(f_{2112}(\lambda) + h_{2112}(\lambda)), \\
h_{2112}(\lambda) &= q^2 \frac{(1 + q^2)}{(1 - q^2)^2} 
\frac{
  (1 + q^{2\lambda + 3})(1 + q^{2\lambda - 1})
}{
  (1 - q^{4\lambda})(1 - q^{4\lambda - 2})
}.
\end{align*}
And therefore we have
\begin{align*}
A^{\lambda} F_3 &\sim (f_{12}(\lambda) - qf_{21}(\lambda)) A^{\lambda} + (g_{12}(\lambda) - q g_{21}(\lambda))A^{\lambda-1}, \\
A^{\lambda} F_1 F_2 F_3 &\sim (f_{1212}(\lambda) - q f_{1221}(\lambda)) A^{\lambda} + (g_{1212}(\lambda) - q g_{1221}(\lambda)) A^{\lambda-1} + (h_{1212}(\lambda) - q h_{1221}(\lambda)) A^{\lambda - 2}, \\
A^{\lambda} F_2 F_1 F_3 &\sim (f_{2112}(\lambda) - q f_{2121}(\lambda)) A^{\lambda} + (g_{2112}(\lambda) - q g_{2121}(\lambda)) A^{\lambda-1} + (h_{2112}(\lambda) - q h_{2121}(\lambda)) A^{\lambda - 2}, \\
A^{\lambda} F_3^2 &\sim (f_{123}(\lambda) - q f_{213}(\lambda)) A^{\lambda} + (g_{123}(\lambda) - q g_{213}(\lambda)) A^{\lambda} + (h_{123}(\lambda) - q h_{213}(\lambda)) A^{\lambda - 2}.
\end{align*}
Hence \eqref{q-radial_eqn_uqsl3_commutanttrick} becomes
\begin{align*}
A^{\lambda} \check{E}_3 F_3 
  &\sim q^{2\lambda - 2} \left( q^{2\lambda+2} f_{33}(\lambda+1) + q^{2\lambda+3} (q - q^{-1}) f_{123}(\lambda+1) + \frac{(c_2 - q^{4} c_2^{-1})}{(q - q^{-1})} f_3(\lambda+1) \right) A^{\lambda+1} \\
  &\qquad + q^{2\lambda - 2} \left( q^{2\lambda + 2} g_{33}(\lambda+1) + q^{2\lambda + 3} (q - q^{-1}) g_{123}(\lambda+1) \right. \\
  &\qquad \qquad \qquad \left. + \frac{(c_2  - q^{4}c_2^{-1})}{(q - q^{-1})} g_3(\lambda+1) + \frac{(q^4 c_2^{-1} - c_2)}{(q - q^{-1})} f_3(\lambda) \right) A^{\lambda} \\
  &\qquad + q^{2\lambda - 2} \left( q^{2\lambda + 2} h_{33}(\lambda+1) + q^{2\lambda + 3} (q - q^{-1}) h_{123}(\lambda+1) + \frac{(q^4 c_2^{-1} - c_2)}{(q - q^{-1})} g_3(\lambda) \right) A^{\lambda-1}.
\end{align*}
Combining the explicit expressions of $A^{\lambda} E_1 F_2$, $A^{\lambda} E_2 F_2$ and $A^{\lambda} \check{E}_3 F_3$ modulo $\sim$ gives
\begin{equation} \label{q-radial_eqn_scalarqdiffUqsl3}
(1 - q^2)^2 A^{\lambda} \widetilde{\Omega}_1 \sim f(q^{\lambda}) (A^{\lambda+1} - A^{\lambda}) + f(q^{-\lambda-2}) (A^{\lambda-1} - A^{\lambda}) + \frac{(1 - q^6)}{(1 - q^2)} A^{\lambda},
\end{equation}
where $f(q^{\lambda})$ is
\begin{equation*}
f(q^{\lambda}) = \frac{
  (1 + c_2 d_2 q^{2\lambda})(1 + c_2^{-1} d_2^{-1} q^{2\lambda + 6})(1 - c_2 d_2^{-1} q^{2\lambda + 4})(1 - c_2^{-1} d_2 q^{2\lambda + 4})
}{
  (1 - q^{4\lambda + 4})(1 - q^{4\lambda + 6})
},
\end{equation*}
and $g(\lambda) = f(-\lambda - 2)$.
Observe that from \eqref{q-radial_eqn_scalarqdiffUqsl3} and the flip operator $\sigma$ we have $A^{\lambda} \widetilde{\Omega}_1 \sim A^{\lambda} \widetilde{\Omega}_2$.
With the identification $z = q^{2\lambda + 2}$ and $A^{\lambda-1}, A^{\lambda}, A^{\lambda+1} \mapsto q^{-1}z, z, q z$ in \eqref{q-radial_eqn_scalarqdiffUqsl3}, the polynomial solutions of the $q$-difference equation \eqref{q-radial_eqn_scalarqdiffUqsl3} are Askey-Wilson polynomials in two free parameters.
The Askey-Wilson polynomials are of the form $p_n(x; -c_2 d_2 q^{-2}, -c_2^{-1} d_2^{-1} q^4, c_2 d_2^{-1} q^2, c_2^{-1} d_2 q^2 | q^2)$.
Upon using $p^{(\alpha, \beta)}_{n}(x;s,t|q) = p_n(x; q^{\frac{1}{2}}ts^{-1}, q^{\frac{1}{2} + \alpha} st^{-1}, -q^{\frac{1}{2}}(st)^{-1}, -q^{\frac{1}{2}+\beta}st | q)$ we find the polynomial solutions for \eqref{q-radial_eqn_scalarqdiffUqsl3} to be
\begin{equation*}
p_n^{(1,0)}(x; q^{-1}c_2, q^{-2}d_2 | q^2) = p_n(x; -c_2 d_2 q^{-2}, -c_2^{-1} d_2^{-1} q^4, c_2 d_2^{-1} q^2, c_2^{-1} d_2 q^2 | q^2).
\end{equation*}
This result extends the classification given by Letzter \cite[Theorem 8.2]{Let04} to an example of a quantum symmetric pair where the restricted root system is non-reduced.
With parametrization $c = q^{\sigma + 1}$ and $d = q^{\tau + 2}$, where $\sigma, \tau \in \RR$, we identify the radial part of the center restricted to the trivial representation with Dijkhuizen and Noumi \cite[Theorem 5.4]{DN98}.
Using the same calculations we can compute the second order $q$-difference equation $\Pi_{t_1, t_2}(\widetilde{\Omega}_i)$ for any irreducible finite dimensional representations $t_1$ and $t_2$ on respectively $\mathcal{B}_{\textbf{c}}$ and $\mathcal{B}_{\textbf{d}}$.
We will not work out the details in this paper.
But it will be interesting if the matrix valued second order $q$-difference equations $\Pi_{t_1, t_2}(\widetilde{\Omega}_i)$ for any irreducible finite dimensional representations $t_1$ and $t_2$ can be related to matrix valued spherical functions and matrix valued orthogonal polynomials.
\end{example}

\subsection*{Acknowledgments}
The author thanks Erik Koelink and Pablo Rom\'an for useful discussions.
Parts of this paper have been discussed with Stefan Kolb, the author thanks him for his input and his hospitality during the visit of the author to Newcastle.
The research of the author is supported by the Netherlands Organization for Scientific Research (NWO) under project number \textbf{613.001.005} and by the Belgian Interuniversity Attraction Pole Dygest \textbf{P07/18}.

\end{document}